\documentclass[12pt,a4paper]{amsart}    

\usepackage{amsmath}
\usepackage{amsfonts}
\usepackage{amssymb}
\usepackage{amsthm}
\usepackage{mathrsfs}
\usepackage{bm}
\usepackage{graphicx}
\usepackage{color}
\usepackage{marginnote}
\usepackage{mathtools}
\usepackage{hyperref}
\usepackage{xcolor}
\usepackage{tikz}

\mathtoolsset{showonlyrefs}
\allowdisplaybreaks[1]
\newtheorem{theorem}{Theorem}
\newtheorem{proposition}[theorem]{Proposition}
\newtheorem{lemma}[theorem]{Lemma}
\newtheorem{corollary}[theorem]{Corollary}

\theoremstyle{definition}

\theoremstyle{remark}
\newtheorem{remark}[theorem]{Remark}

\numberwithin{equation}{section}
\numberwithin{theorem}{section}

\newcommand{\re}{\mathbb R}
\newcommand{\R}{\mathbb R}
\newcommand{\Nt}{\mathbb N}
\newcommand{\Lpu}{L_{\rm{uloc}}^{p}}

\newcommand{\OPL}{{\Del - A\Del K \! \ast}}         
\newcommand{\OPs}{{ A\Del K \! \ast}}         
\newcommand{\N}{\nabla}
\newcommand{\pt}{\partial}   

\newcommand{\lam}{\lambda}
\newcommand{\ep}{\varepsilon}

\newcommand{\Del}{\Delta}

\newcommand{\dx}{\,{\rm d}x}
\newcommand{\dy}{\,{\rm d}y}

\newcommand{\dta}{\,{\rm d}\tau}
\newcommand{\ds}{\,{\rm d}s}
\renewcommand{\d}[1]{{\rm d}#1}

\newcommand{\rdd}{\color{blue}}

\newcommand{\eq}[1]{\begin{equation}#1\end{equation}}%
\newcommand{\spl}[1]{{\begin{split}#1\end{split}}}
\newcommand{\eqn}[1]{\begin{equation*}#1\end{equation*}}
\newcommand{\eqsp}[1]{\begin{equation}\begin{split}#1\end{split}\end{equation}}
\newcommand{\eqspn}[1]{\begin{equation*}\begin{split}#1	\end{split}\end{equation*}}%

\newcommand{\uloc}[1]{#1,{\rm uloc}}
\newcommand{\git}{global-in-time }

\newcommand\blfootnote[1]{\begingroup\renewcommand\thefootnote{}\footnote{#1}\addtocounter{footnote}{-1}\endgroup}

\definecolor{lime}{HTML}{A6CE39}
\DeclareRobustCommand{\orcidicon}{%
	\begin{tikzpicture}
		\draw[lime, fill=lime] (0,0) 
		circle [radius=0.16] 
		node[white] {{\fontfamily{qag}\selectfont \tiny ID}};
		\draw[white, fill=white] (-0.0625,0.095) 
		circle [radius=0.007];
	\end{tikzpicture}
	\hspace{-2mm}
}

\begin{document}

\title{Stability of constant steady states of a chemotaxis model  
}


%
%

\author[S. Cygan]{Szymon Cygan \href{https://orcid.org/0000-0002-8601-829X}{\orcidicon}}
\address[S. Cygan]{	Instytut Matematyczny, Uniwersytet Wroc\l{}awski, pl. Grunwaldzki 2/4, \hbox{50-384} Wroc\l{}aw, Poland \\ \href{https://orcid.org/0000-0002-8601-829X}{orcid.org/0000-0002-8601-829X}}
\email[S. Cygan]{szymon.cygan2@uwr.edu.pl}

\author[G. Karch]{Grzegorz Karch \href{https://orcid.org/0000-0001-9390-5578}{\orcidicon}}
\address[G. Karch]{	Instytut Matematyczny, Uniwersytet Wroc\l{}awski, pl. Grunwaldzki 2/4, \hbox{50-384} Wroc\l{}aw, Poland \\ 
\href{https://orcid.org/0000-0001-9390-5578}{orcid.org/0000-0001-9390-5578}}
\email[G. Karch]{grzegorz.karch@uwr.edu.pl}

\author[K. Krawczyk]{Krzysztof Krawczyk \href{https://orcid.org/0000-0003-1365-038X}{\orcidicon}}
\address[K. Krawczyk]{	Instytut Matematyczny, Uniwersytet Wroc\l{}awski, pl. Grunwaldzki 2/4, \hbox{50-384} Wroc\l{}aw, Poland \\ 
\href{https://orcid.org/0000-0003-1365-038X}{orcid.org/0000-0003-1365-038X}}
\email[S. Cygan]{krzysztof.krawczyk@uwr.edu.pl}

\author[H. Wakui]{Hiroshi Wakui \href{https://orcid.org/0000-0002-4676-4669}{\orcidicon}}
\address[H. Wakui]{Faculty of Science Division I, Tokyo University of Science, 1-3 Kagurazaka, Shinjuku-ku, Tokyo \hbox{162-8601}, Japan \\
\href{https://orcid.org/0000-0002-4676-4669}{orcid.org/0000-0002-4676-4669}}
\email[H. Wakui]{hiroshi.wakui@rs.tus.ac.jp}

\date{\today}

\begin{abstract}
	The Cauchy problem for the parabolic--elliptic Keller--Segel system in the whole $n$-dimensional space is studied. For this model, every constant $A \in \mathbb{R}$ is a stationary solution. The main goal of this work is to show that $A < 1$ is a stable steady state  while $A > 1$ is unstable.  Uniformly local Lebesgue spaces are used in order to deal with solutions that do not decay at spatial variable on the unbounded domain. 
\end{abstract}



\maketitle

\blfootnote{This paper was invented and written online during the 2020 lockdown.}


\section{Introduction}
\label{intro}

There are several mathematical works on
the chemotaxis model   introduced by Keller
and Segel~\cite{KeSi}.  Here, 
we refer the reader only to the monograph \cite{yagi} and the reviews~\cite{BBTW,H} for a discussion of those mathematical results as well as for additional references.
In this paper, 
we consider the following minimal parabolic-elliptic Keller-Segel system
\eq{
	u_t - \Delta u + \nabla \cdot (u \nabla \psi) = 0, \quad 
	- \Delta \psi + \psi = u \quad \text{for} \quad  t>0,\quad x \in \re^{n},
	\label{eq:IntDD}
}
where $u=u(t,x)$ denotes the density of cells and $\psi=\psi(t,x)$ is a~concentration of chemoattractant. In these equations, all constant  parameters are equal to one for simplicity of the exposition.
System \eqref{eq:IntDD} was already studied  in the whole space 
{\it e.g.} in the papers~\cite{BCKZ,BGK,CPZ,KSS,KS11,KSJ,KSY,RCD}, where several results either on a~blow up or on a large time behavior of solutions  have been obtained.

For each constant $A \in \re$, the couple $(u, \psi) = (A,A)$ is a stationary solution of  system~\eqref{eq:IntDD} and, since the domain is unbounded, it does not belong to any Lebesgue $L^p$-space with $p\in [1,\infty)$.
Thus, in Theorem \ref{thm;main} and in Section~\ref{sec:LiTS}, we develop a mathematical theory concerning  local-in-time solutions to the initial value problem for system~\eqref{eq:IntDD} in the uniformly local Lebesgue spaces~$\Lpu(\re^{n})$.
Then, we consider a constant  stationary solution $(u, \psi) = (A,A)$
with $A\in [0,1)$
and we show in Theorem~\ref{thm;global-exist} that  a small $L^p$-perturbation of such an initial datum gives a global-in-time solution which converges toward $(A,A)$ as $t\to\infty$. On the other hand, we prove in the Theorem~\ref{thm;instability} that the constant solution is unstable in the Lyapunov sense if~$A>1$.

A stability of constant solutions of chemotaxis models has been already studied in bounded domains. For example, the paper \cite{GuoHw} describes dynamics near an unstable constant solution to the classical parabolic-parabolic Keller-Segel model in a bounded domain and obtained results are interpreted as an early pattern formation. Another work \cite{Wink} is devoted to the system
\eqn{
	u_t - \Delta u + \nabla \cdot (u \nabla \psi) = 0, \quad 
	- \Delta \psi + \mu = u, \quad 
	\mu \equiv \frac{1}{|B_{R}(0)|} \int_{B_{R}(0)} u \, {\rm d}x, 
}
in the ball of radius $R>0$ with Neumann boundary condition. Here, constants are also stationary solutions and it is shown 
in the work \cite{Wink}
that there exists a critical number~$m_c$ such that at mass levels above $m_c$ the constant steady states are extremely unstable and blow-up can occur. On the other hand, for $m < m_c$ there exist infinitely many radial solutions with a~mass equal to $m$.

\subsection*{Notation}
The usual norm of the Lebesgue space 
$L^p (\re^n)$ 
with respect to the spatial variable is denoted by $
\|\cdot\|_p$ 
for all 
$p \in [1,\infty]$.
In the following, we use also the uniformly local Lebesgue space $\Lpu(\re^n)$ with the norm $\| \cdot \|_{\uloc{p}}$ defined below by formula  \eqref{eq;lpulocnorm}.
Any other norm in a Banach space $Y$ is denoted by $\|\cdot\|_Y$. 
The letter $C$ corresponds to a generic constant (always independent of $t$ and $x$) which may vary from line to line. 
We write
$C=C(\alpha,\beta,\gamma, ...)$ when we want to emphasize the dependence of $C$ on parameters~$\alpha,\beta,\gamma, ...$~. 
We use standard definition of Fourier transform 
$\widehat{f}(\xi) = (2\pi)^{-n/2}\int_{\re^n} e^{- i \xi \cdot x} f(x) \, \d{x}$.

\section{Results and comments}
\label{sec:RnC}
Our goal is to study properties of solutions to the Cauchy problem for the simplified parabolic--elliptic Keller--Segel model of chemotaxis 
\begin{equation} \label{eq;DD1}
	\left\{
	\spl{
		& u_t - \Del u + \N \cdot (u \N\psi) = 0,&&\quad t > 0,\ \  x \in \re^{n},\\
		&-\Del \psi + \psi = u,&&\quad t>0,\ \ x \in \re^{n},\\
		&u(0,x)=u_{0}(x),&& \quad x \in \re^{n},
	}
	\right.
\end{equation}
with  $n \ge 1$. We solve the second equation with respect to $\psi$ to obtain $\psi = K *u$, where $K$~is the Bessel function (see Lemma \ref{lem:KProp} below) which reduces 
problem~\eqref{eq;DD1}  to the following one
\begin{equation}\label{eq;DD} 
	\left\{
	\spl{
		&u_t - \Delta u + \nabla \cdot (u \nabla K* u) = 0,&&\quad t > 0,\ \  x \in \re^{n}, \\
		&u(0,x)=u_{0}(x),&& \quad x \in \re^{n}.
	} 
	\right.
\end{equation}

We begin by a result on an existence of local-in-time solutions to problem~\eqref{eq;DD} in the uniformly local Lebesgue spaces
\begin{equation}\label{eq;lpulocnorm}
\begin{split}
	\Lpu(\re^{n})
	\equiv
	\bigg\{
	f \in L^{p}_{\text{loc}}&(\re^{n}) 
	:
	\| f \|_{\uloc{p}}
	\equiv
	\sup_{x \in \re^{n}}
	\left(
	\int_{B_{1}(x)}
	|f(y)|^{p}
	\,
	\dy
	\right)^{1/p}<+\infty
	\bigg\}
	\end{split}
\end{equation}
for $p\in [1,\infty)$ and $L^\infty_{\rm uloc}(\re^n)=L^\infty (\re^n)$.

\begin{theorem}\label{thm;main}
	For each $p$ satisfying
	\begin{equation}\label{eqn:warP}
		\begin{split}
			p \in \left[ 1,  \infty\right] \quad & \text{if} \quad n = 1, \\
			p \in \left[ \frac{3}{2}, \infty\right] \quad & \text{if} \quad n = 2, \\
			p\in \left( \frac{n}{2}, \infty\right] \quad & \text{if} \quad n \ge 3,
		\end{split}
	\end{equation} 
	and every $u_{0} \in \Lpu(\R^n)$,
	there exists $T>0$ and a unique mild solution 
	\eqspn{u \in L^{\infty}\big([0,T);\Lpu(\re^{n})\big)
		\cap C\big((0,T);\Lpu(\re^{n})\big)}
	of problem \eqref{eq;DD}. Moreover, if $u_{0} \ge 0$, then  $u(t,x) \ge 0$ almost everywhere in~$[0,T) \times \re^{n}$.
\end{theorem}

The more-or-less standard proof of Theorem~\ref{thm;main} is based on the Banach contraction principle applied to an integral representation of  solutions to problem~\eqref{eq;DD} (see Section~\ref{sec:LiTS} for more details).

Now, we formulate a simple consequence of Theorem \ref{thm;main} in the case when an initial condition is an $L^p$-perturbation of a constant $A\in \R$.

\begin{corollary}
	\label{lem;LitSol}
	Let 
	$p$ satisfy conditions \eqref{eqn:warP}. 
	For every $A \in \R$ and every $v_0 \in L^p(\R^n)$
	there exists a unique local-in-time mild solution $u=u(t,x)$ of problem \eqref{eq;DD} 
	(as stated in Theorem \ref{thm;main}) corresponding to the initial datum 
	$u_0 = A + v_0 \in \Lpu (\R^n)$. 
	This solution satisfies 
	$u - A \in C\big([0,T );L^p(\R^n)\big)$.
\end{corollary}

This corollary is an immediate  consequence of the uniqueness of solutions established in Theorem~\ref{thm;main} combined with the uniqueness result of solutions to the perturbed problem considered in Proposition \ref{prop;local-exist}, below.
 
Next,
we show that one can  construct \git solutions around each constant solution  $A \in [0, 1)$.   
\begin{theorem}\label{thm;global-exist}
	Let $A \in [0, 1)$. Assume that the exponent $p$ satisfies conditions~\eqref{eqn:warP} and moreover $p \le n$. Fix $q\in (n,2p]$.
	 There exists $\varepsilon >0$ such that for every $v_0 \in L^p(\re^n)$ with $\| v_0\|_{p} < \varepsilon$, problem~\eqref{eq;DD} with the initial condition $u_0 = A + v_0$ has a unique \git mild solution $u(t,x)$ satisfying 
	$
	u-A \in C\big([0,\infty);L^p(\R^n)\big) 
	$ 
	 and 
	$$ 
	 \| u(t) - A \|_{p}+
	 t^{\frac{n}{2}(\frac{1}{p} - \frac{1}{q})} \| u(t) - A \|_{q} 
	 \leq C \| u_0 - A \|_{p}
	$$
	for a constant $C>0$ and all $t>0$.
\end{theorem}

The smallness assumption imposed on initial conditions in Theorem~\ref{thm;global-exist} seems to be necessary. This is clear in the case $A = 0$, where sufficiently large initial data lead to solutions which blow-up in finite time, see \textit{e.g.} \cite{BCKZ,CPZ,KS11,KSJ}  for blow-up results for solutions of system \eqref{eq:IntDD} considered in the whole space.

Next, we deal with $A>1$ which appears to be the unstable constant stationary solution.

\begin{theorem}\label{thm;instability}
	The constant stationary solution $A>1$ of problem \eqref{eq;DD} is not stable in the Lyapunov sense under small perturbations from $L^{p}(\R^n)$ for each $p$ satisfying condition~\eqref{eqn:warP} except $p = 1$ and $p = \infty$.
\end{theorem}

In this theorem, we do not claim that solutions corresponding to $L^p$-per\-tur\-ba\-tions of $A > 1$ are global-in-time. We show only that if they are \git then they cannot be stable.

To conclude this section, we notice that a constant $A < 0$ is linearly stable which we comment in Remark \ref{rem;Abelow0}, below. The proof of nonlinear stability of this constant can be obtained by the method used in the proof of Theorem \ref{thm;global-exist}. We add some comments on the linear stability of the constant solution $A = 1$ in Remark \ref{rem;stabA1}, below.


\section{Local-in-time solutions in uniformly local Lebesgue spaces}
\label{sec:LiTS}

We find solutions to problem \eqref{eq;DD} via its formulation in the integral form
\eq{\label{eq;IE}
	u(t)
	=
	e^{t\Del}u_{0}
	-
	\int_{0}^{t}
	\N e^{(t-s)\Del}
	\cdot \big(u(s) \N K* u(s)\big)
	\, \d{s}
}
with the heat semigroup given by the formula
$$
(e^{t\Del}f)(x)
\equiv
(4\pi t)^{-\frac{n}2}
\int_{\re^{n}}
e^{-\frac{|x-y|^{2}}{4t}}
f(y)
\,
\dy.
$$
This construction requires  auxiliary results which we are going to gather and prove below. Then, Theorem~\ref{thm;main} is proved at the end of this section. We begin by recalling properties of the heat semigroup $\{ e^{t\Delta } \}_{t \ge 0}$ acting on the $\Lpu$-spaces. 

\begin{proposition}[\cite{Ar-Ro-Ch-Dl,Ma-Te}] \label{prop;heat-semi-uloc}
	
	For all $1 \le q \le p \le \infty$, $k \in \mathbb{Z}_{+}$ and $\alpha \in \mathbb{Z}_{+}^{n}$ there exist numbers
	$C=C(n,p,q,k,\alpha) >0$
	such that
	\eq{\label{eq;heat-semi-Lp-Lq}
		\| \pt_{t}^{k}\pt_{x}^{\alpha}{e}^{t\Del}f \|_{\uloc{p}}
		\le
		C t^{-k-\frac{|\alpha|}2}
		\big(
		1
		+
		{t^{-{\frac{n}2}( \frac1q-\frac1p)}}
		\big)
		\| f \|_{\uloc{q}}
	}
	for every $f \in \Lpu(\re^{n})$ and all $t>0$. 
	In particular, when $p=q$ and $1 \le p \leq \infty$, it holds that
	\eq{\label{eq;heat-semi-Lp-Lp}
	 	\spl{
			\| {e}^{t\Del}f \|_{\uloc{p}}
			\le 
			\| f \|_{\uloc{p}} \qquad \text{for all} \quad t\geq 0.
 }
}
\end{proposition}

\begin{remark}\label{rem:heatest}
Proposition \ref{prop;heat-semi-uloc} generalizes the following well-known estimates of the heat semigroup acting on the Lebesgue space $L^p(\R^n)$
	\eqspn{\label{eq;heat-semi-Lp-Lq-normal}
		\| \pt_{t}^{k}\pt_{x}^{\alpha}{e}^{t\Del}f \|_{p}
		\le
		C {t^{-{\frac{n}2}( \frac1q-\frac1p)-k-\frac{|\alpha|}2}}
		\| f \|_{q} 
		\quad 
		\text{for all} 
		\quad t>0. 
		}
\end{remark}

\begin{remark}\label{prop;uloc-equivalent-cond-closure}
	It follows from estimates \eqref{eq;heat-semi-Lp-Lq} applied with $k=0$, $\alpha =0$ and $p=q$ that $e^{t\Delta} f \in L^\infty\big( (0, \, \infty) ; \Lpu(\R^n)\big)$ for each $f\in \Lpu(\R^n)$ and, in general, this mapping in not continuous in time. This continuity holds true in the smaller space
	$$
	\mathcal{L}_{\mathrm{uloc}}^{p}(\re^{n})
	=
	\overline{BUC(\re^{n})}^{\| \cdot \|_{\uloc{p}}},
	$$
	where $BUC(\re^{n})$ is a space of all bounded uniformly continuous functions on~$\re^{n}$. In fact, the following statements are equivalent:
	\begin{enumerate}
		\item $f \in \mathcal{L}_{\mathrm{uloc}}^{p}(\re^{n})$.
		\item 
		$
			\| f(\cdot + h) - f \|_{\uloc{p}} \to 0 $ as  $|h| \to 0.
		$
		\item
		$
		\|
		e^{ t \Del}f 
		-f
		\|_{\uloc{p}}
		\to
		0$ as $t \to 0$.
	\end{enumerate}
	We refer the reader to \cite[Proposition 2.2]{Ma-Te} for the proof of these properties.
\end{remark}

\begin{remark}\label{lem;uloc-norm-scaling} 
	There is an alternative definition of the norm in the uniformly local $L^p$-spaces
	$$
	\| f \|_{\uloc{p},\rho}
	\equiv
	\sup_{x \in \re^{n}}
	\Big(
	\int_{B_{\rho}(x)}
	|f(y)|^{p}
	\,
	\dy
	\Big)^{1/p}
	$$
	for each $\rho > 0$. However, by a simple scaling property, we can show that all these norms are in fact equivalent.
	
\end{remark}

Let us also recall properties of the Bessel kernel which are systematically used in this work.
\begin{lemma}\label{lem:KProp}
	Denote by $\psi \in \mathcal{S}'(\re^n)$ a solution of the equation $-\Delta\psi + \psi = u$ for some $u \in \mathcal{S}'(\re^n)$. The following statements hold true.
	\begin{enumerate}
		\item $\psi = K*u$, where $\widehat{K}(\xi) = \frac{1}{1+|\xi|^2}$.
		\item For $n=1$, $K(x) = \frac{1}{2} e^{-|x|}$.
		\item For $n\ge 2$, $K\in L^1(\R^n) \cap L^p(\R^n)$ for each $p\in \big[1, \, \frac{n}{n-2} \big)$ and $\N K \in L^1(\R^n) \cap L^q(\R^n)$ for each $q \in \big[ 1,\frac{n}{n-1}\big)$.
		\item $|\N K(x)|=|K'(|x|)|=-K'(|x|)$ is radially symmetric and decreasing in $|x|$.
	\end{enumerate}
\end{lemma}
\begin{proof}[Sketch of the proof]
	Property 1 is well known. 
	Item 2 can be obtained by a direct calculation. 
	For the proofs of properties  3 and 4 we refer the reader to \cite[Sec.~1.2.5.]{Ad-He}. 
	In particular, in order to show property 4, we recall that $K(x)=(2\pi)^{-\frac{n}2}|x|^{-\frac{n-2}2}K_{\frac{n-2}2}(|x|)$,
	where $K_{\nu}=K_{\nu}(r) \ (r >0, \nu > -\frac12)$ stands for the modified Bessel function of the third kind, that is
	$$
	K_{\nu}(r)
	=
	\frac{\pi^{\frac12}}{2^{\frac12}\Gamma( \nu+\frac12 )}
	r^{-\frac12}
	\mathrm{e}^{-r}
	\int_{0}^{\infty}
	s^{\nu-\frac12}
	\left(
	1
	+
	\frac{s}{2r}
	\right)^{\nu-\frac12}
	\mathrm{e}^{-s}
	\,
	\ds,
	$$
	see \textit{e.g.} \cite[formula (1.2.25), p.12]{Ad-He}. In particular, we observe that
	\eqspn{
		K'(r)
		=
		-
		\frac{1}{4\pi} 
		\Big(
		\frac{2}{\pi}
		\Big)^{\frac12}
		r^{-\frac{n-2}{2}}
		K_{\frac{n}2}(r)
		.}
	Since $K_{\frac{n}2}(|x|)>0$ and $|x|^{-\frac{n-2}2}$ decreases in $|x|$,
	the function $-|x|^{-\frac{n-2}2}K_{\frac{n}2}(|x|) < 0$ increases in $|x|$.
	Therefore
	$
	|\N K(x)|
	=
	|K'(|x|)|
	=
	-K'(|x|)
	$
	is radially symmetric and decreasing in $|x|$.
\end{proof}

Our goal is to estimate the nonlinear term on the right hand side of equation~\eqref{eq;IE}. We begin with  the following elementary lemma.
\begin{lemma}\label{lem;ball-qube-equiv}
	For $f \in \Lpu(\R^n)$ it holds that
	$$
	3^{-n}
	\sup_{x \in \re^{n}}
	\| f \|_{L^{p}(B_{1}(x))}
	\le
	\sup_{k \in \mathbb{Z}^{n}}
	\| f \|_{L^{p}(Q_{\frac12}(k))}
	\le
	2^{n}
	\sup_{x \in \re^{n}}
	\| f \|_{L^{p}(B_{1}(x))},
	$$
	where $Q_{r}(k) \equiv \left\{ (x_{1}, \cdots, x_{n}) \in \re^{n} : \max_{1 \le j \le n} |x_{j}-k_{j}| \le r \right\}$
	for 
	$k = (k_{1}, \cdots, k_{n}) \in \mathbb{Z}^{n}$
	and
	$r>0$.
\end{lemma}
\begin{proof}
	Obviously, it holds that 
	$
	\re^{n}
	=
	\bigcup_{k \in \mathbb{Z}^{n}}
	Q_{1/2}(k)
	$
	with $|Q_{1/2}(k) \cap Q_{1/2}(\ell)|=0$ for $k \not = \ell$.
	Thus, for $x \in \re^{n}$, 
	there exists $k^{0} \in \mathbb{Z}^{n}$ such that
	$x \in Q_{1/2}(k^{0})$.
	Since $x \in Q_{1/2}(k^{0})$, there exist 
	lattice points $\{ k^{\ell}\}_{\ell=1}^{3^{n}-1} \subset \mathbb{Z}^{n}$ such that
	$B_{1}(x) \subset \bigcup_{\ell=0}^{3^{n}-1} Q_{1/2}(k^{\ell})$.
	Indeed,
	the set 
	\eqn{
		P
		\equiv
		\left\{ 
		k \in \mathbb{Z}^{n} 
		: 
		\spl{
			&\text{There exists } 
			\ell \in \{ 1, \cdots, n\} 
			\text{ and } 
			\{ m_{j} \}_{j=1}^{\ell} \subset \{ 1, \cdots, n\}
			\text{ such that }\\
			&
			m_{j} < m_{j+1},
			{e}_{m_{j}}
			=
			(0, \cdots, \underbrace{1}_{m_{j}\text{-th}}, \cdots, 0)
			\text{ and }
			k
			=
			k^{0}
			\pm
			{e}_{m_{1}}
			\pm 
			\cdots 
			\pm
			e_{m_{\ell}}
		}
		\right\},
	}
	consists of $(3^{n}-1)$ points in $\mathbb{Z}^{n}$.
	Therefore the number of elements of the set $\{ k^{0} \} \cup P$ is~$3^{n}$, and thus
	we have
	\eqsp{
	\Big(
	\int_{B_{1}(x)}
	|f(y)|^{p}
	\,
	\dy
	\Big)^{\frac1p}
	\le
	\sum_{j=0}^{3^{n}-1}
	\Big(
	\int_{Q_{\frac12}(k^{j})}
	|f(y)|^{p}
	\,
	\dy
	\Big)^{\frac1p}
	\le
	3^{n}
	\sup_{k \in \mathbb{Z}^{n}}
	\Big(
	\int_{Q_{\frac12}(k)} 
	|f(y)|^{p}
	\,
	\dy
	\Big)^{\frac1p}.
	}
	Hence we obtain the first part of our claim.
	On the other hand, it holds that
	$
	Q_{1/2}(k)
	\subset
	\bigcup_{j=1}^{2^{n}}
	B_{1}(x^{j}),
	$
	where
	$\{ x^{j} \}_{j=1}^{2^{n}} \subset \re^{n}$
	with 
	$x^{j}_{\ell}=k_{\ell}\pm \frac12$ for all $\ell=1, \cdots, n$ and $j=1, \cdots, 2^{n}$.
	Thus we observe 
	\eqsp{
	\Big(
	\int_{Q_{\frac12}(k)} 
	|f(y)|^{p}
	\,
	\dy
	\Big)^{\frac1p}
	\le
	\sum_{j=1}^{2^{n}}
	\Big(
	\int_{B_{1}(x^{j})}
	|f(y) |^{p}
	\,
	\dy
	\Big)^{\frac1p} 
	\le
	2^{n}
	\sup_{x \in \re^{n}}
	\Big(
	\int_{B_{1}(x)}
	|f(y) |^{p}
	\,
	\dy
	\Big)^{\frac1p}.
	}
\end{proof}

Now, we prove a Young type inequality in the space $\Lpu(\re^n)$.

\begin{proposition}\label{lem;bessel-gradient-estimate}
	Let $1+\frac1p=\frac1q+\frac1r$ with $1 \le q \le p \le \infty$ and $1 \le r < \frac{n}{n-1}$. 
	Then, there exists a number $C = C(n)>0$ such that
	\eq{\label{eq;bessel-gradient-estimate}
		\| \N K * f \|_{\uloc{p}}
		\le
		C \big(
		\| \N K \|_{1}
		+
		\| \N K \|_{r}
		\big)
		\| f \|_{\uloc{q}}.
	}
\end{proposition}
\begin{proof}
	This lemma has been proved in \cite[Theorem 3.1]{Ma-Te} in the case of an arbitrary bounded and integrable function $K$. Since the Bessel kernel is not bounded for $n\ge 2$, we revise those arguments in our setting.
	Using
	the decomposition
	$\re^{n}=\bigcup_{k \in \mathbb{Z}^{n}} Q_{1/2}(k)$ and
	\eqn{
		\N K(x)
		=
		\sum_{k \in \mathbb{Z}^{n}}
		\chi_{Q_{\frac12}(k)}(x)
		\N K(x)
		,\quad
		f(x)
		=
		\sum_{k \in \mathbb{Z}^{n}}
		\chi_{Q_{\frac12}(k)}(x)
		f(x)
		\qquad
		\text{ a.e. }
		x \in \re^{n},
	}
	we observe
	\eq{\label{eq;bessel-gradient-estimate1}
		\spl{
			\| \N K * f \|_{L^{p}(Q_{\frac12}(k))}
			=&
			\bigg\|
			\bigg(
			\sum_{k' \in \mathbb{Z}^{n}}
			\chi_{Q_{\frac12}(k')}
			\N K
			\bigg)
			*
			\bigg(
			\sum_{k'' \in \mathbb{Z}^{n}}
			\chi_{Q_{\frac12}(k'')}
			f
			\bigg)
			\bigg\|_{L^{p}(Q_{\frac12}(k))}\\
			=&
			\bigg\|
			\sum_{k', k'' \in \mathbb{Z}^{n}}
			\big(
			\chi_{Q_{\frac12}(k')}
			\N K
			\big)
			*
			\big(
			\chi_{Q_{\frac12}(k'')}
			f
			\big)
			\bigg\|_{L^{p}(Q_{\frac12}(k))}\\
			\le&
			\sum_{k', k'' \in \mathbb{Z}^{n}}
			\big\|
			\big(
			\chi_{Q_{\frac12}(k')}
			\N K
			\big)
			*
			\big(
			\chi_{Q_{\frac12}(k'')}
			f
			\big)
			\big\|_{L^{p}(Q_{\frac12}(k))}.
		}
	}
	The well-known fact that
	$\mathrm{supp} \, f_{1}*f_{2} \subset \mathrm{supp} \, f_{1}+\mathrm{supp} \, f_{2}$
	implies that
	the support of 
	the function 
	$
	(
	\chi_{Q_{1/2}(k')}
	\N K
	)
	*
	(
	\chi_{Q_{1/2}(k'')}
	f
	)
	$
	lies in the cube $Q_{1}(k'+k'')$.
	Thus for $1 \le p, q, r \le \infty$ with $1+\frac{1}{p}=\frac1q+\frac1r$, 
	the Young inequality shows that
	\eq{\label{eq;bessel-gradient-estimate2}
		\spl{
			&
			\sum_{k', k'' \in \mathbb{Z}^{n}}
			\big\|
			\big(
			\chi_{Q_{\frac12}(k')}
			\N K
			\big)
			*
			\big(
			\chi_{Q_{\frac12}(k'')}
			f
			\big)
			\big\|_{L^{p}(Q_{\frac12}(k)\cap Q_{1}(k'+k''))}\\
			\le&
			\sum_{\substack{k', k'' \in \mathbb{Z}^{n}\\ {\underset{{1 \le j \le n}}{\max}} |k_{j}'+k_{j}''-k_{j}| \le 1}}
			\big\|
			\big(
			\chi_{Q_{\frac12}(k')}
			\N K
			\big)
			*
			\big(
			\chi_{Q_{\frac12}(k'')}
			f
			\big)
			\big\|_{L^{p}(Q_{\frac12}(k)\cap Q_{1}(k'+k''))}\\
			\le&
			\sum_{\substack{k', k'' \in \mathbb{Z}^{n}\\ {\underset{{1 \le j \le n}}{\max}} |k_{j}'+k_{j}''-k_{j}| \le 1}}
			\big\|
			\big(
			\chi_{Q_{\frac12}(k')}
			\N K
			\big)
			*
			\big(
			\chi_{Q_{\frac12}(k'')}
			f
			\big)
			\big\|_{p}\\
			\le&
			\sum_{\substack{k', k'' \in \mathbb{Z}^{n}\\ {\underset{{1 \le j \le n}}{\max}} |k_{j}'+k_{j}''-k_{j}| \le 1}}
			\|
			\chi_{Q_{\frac12}(k')}
			\N K
			\|_{r}
			\|
			\chi_{Q_{\frac12}(k'')}
			f
			\|_{q}.
		}
	}
	The set 
	$
	P(k)
	\equiv
	\{
	(k', k'') \in \mathbb{Z}^{n}\times \mathbb{Z}^{n}
	\, | \,
	\max_{1 \le j \le n} |k_{j}'+k_{j}''-k_{j}| \le 1
	\}
	$
	for $k = (k_{1}, \cdots, k_{n}) \in \mathbb{Z}^{n}$
	consists of $3^{n}$ elements.
	Therefore we obtain 
	\eq{\label{eq;bessel-gradient-estimate3}
		\spl{
			\sum_{(k',k'') \in P(k)}
			\|
			\chi_{Q_{\frac12}(k')}
			\N K
			\|_{r}
			\|
			\chi_{Q_{\frac12}(k'')}
			f
			\|_{q}
			\le 
			3^{n}
			\sup_{k'' \in \mathbb{Z}^{n}}
			\|
			f
			\|_{L^{q}(Q_{\frac12}(k''))}
			\sum_{k' \in \mathbb{Z}^{n}}
			\|
			\N K
			\|_{L^{r}(Q_{\frac12}(k'))}.
		}
	}
	
	Due to the compactness of $Q_{1/2}(k)$ for $k \in \mathbb{Z}^{n}$,
	there exists $y_{k} \in Q_{1/2}(k)$ such that
	$
	|y_{k}|
	=
	\mathrm{dist}(Q_{1/2}(k), \{ 0 \}).
	$
	If
	$
	k
	=
	(k_{1}, \cdots, k_{n}) 
	\in \mathbb{Z}^{n}
	$
	satisfies
	$
	\max_{1 \le j \le n} |k_{j}| \ge 2
	$,
	it holds that $0 \not \in Q_{1/2}(k)$.
	We notice that, for $k \in G\equiv\{ \ell \in \mathbb{Z}^{n} \, | \, \max_{1 \le j \le n} |\ell_{j}| \ge 2 \}$, 
	the function $|\N K(x)|$ attains its maximum value at $x=y_{k}$ and we obtain
	$
	\| 
	\N K
	\|_{L^{r}(Q_{1/2}(k))}
	\le
	| \N K(y_{k}) |.
	$
	Thus we have
	\eq{\label{eq;bessel-gradient-estimate4}
		\spl{
			\sum_{k \in \mathbb{Z}^{n}}
			\|
			\N K
			\|_{L^{r}(Q_{\frac12}(k))}
			=&
			\sum_{k \in G}
			\| 
			\N K
			\|_{L^{r}(Q_{\frac12}(k))}
			+
			\sum_{k \in G^{c}}
			\| 
			\N K
			\|_{L^{r}(Q_{\frac12}(k))}\\
			\le&
			\sum_{k \in G}
			|\N K(y_{k})|
			+
			\sum_{k \in G^{c}}
			\| 
			\N K
			\|_{L^{r}(Q_{\frac32}(0))}
			\\
			=&
			\sum_{k \in G}
			|\N K(y_{k})|
			+
			3^{n}
			\| 
			\N K
			\|_{L^{r}(Q_{\frac32}(0))}.
		}
	}
	We define cubes corresponding to $y_{k} \in Q_{1/2}(k)$ for $k \in G$.
	Hence,
	there exists 
	$\ell \in \mathbb{N}$
	and 
	$\{ k^{j} \}_{j=1}^{\ell} \subset \mathbb{Z}^{n}$ 
	such that
	$
	Q_{1/2}(k^{j})
	\cap
	\{ 
	\lambda y_{k}
	:
	0 < \lam < 1
	\}
	\not 
	= 
	\emptyset
	$
	for $j = 1, \cdots, \ell$.
	Here, we denote by $\{ k^{j}\}_{j=1}^{\ell}$  the lattice points such that 
	for 
	$
	0
	=
	\lam_{0}
	<
	\lambda_{1}
	<
	\lambda_{2}
	< 
	\cdots 
	< 
	\lambda_{\ell-1} 
	< 
	\lambda_{\ell} 
	= 1
	$,
	$
	Q_{1/2}(k^{j})^{\circ} \cap \{ \lambda y_{k} : 0 \le \lam \le 1\}
	= 
	\{ \lambda y_{k} : \lambda_{j-1} < \lam < \lam_{j}\}
	$.
	We call such a cube the corresponding cube 
	and
	we denote 
	\eqn{
		P_{k}
		\equiv
		\{ 
		k' \in \mathbb{Z}^{n}
		\, | \, 
		Q_{\frac12}(k') \text{ is a corresponding cube such that }
		Q_{\frac12}(k') = Q_{\frac12}(k)
		\}
	}
	for $k \in G$.
	Since $|\N K(x)|$ decreases in $|x|$ by Lemma \ref{lem:KProp} and $x \in Q_{1/2}(k^{j})$ satisfies $|x| \le |y_{k}|$,
	it holds that
	$$
	|\N K(y_{k})| = \int_{Q_{\frac12}(k^{j})} |\N K(y_{k})| \, \dx
	\le \int_{Q_{\frac12}(k^{j})} |\N K(x)| \, \dx .
	$$
	By \cite[line 5 from above, p.384]{Ma-Te},
	for each fixed $k \in G$,
	the set
	$
	P_{k}
	$
	consists of at most $5^{n}$ points.
	Therefore,  
	\eq{\label{eq;bessel-gradient-estimate5}
		\spl{
			\sum_{k \in G}
			|\N K(y_{k})|
			+
			3^{n}
			\| 
			\N K
			\|_{L^{r}(Q_{\frac32}(0))}
			\le&
			\sum_{k \in G}
			\sum_{k' \in P_{k}}
			\int_{Q_{\frac12}(k')}
			|\N K(x)|
			\,
			\dx
			+
			3^{n}
			\| 
			\N K
			\|_{r}
			\\
			=&
			\sum_{k \in G}
			\sum_{k' \in P_{k}}
			\int_{Q_{\frac12}(k)}
			|\N K(x)|
			\,
			\dx
			+
			3^{n}
			\| 
			\N K
			\|_{r}
			\\
			\le&
			5^{n}
			\sum_{k \in G}
			\int_{Q_{\frac12}(k)}
			|\N K(x)|
			\,
			\dx
			+
			3^{n}
			\| 
			\N K
			\|_{r}
			\\
			\le&
			5^{n}
			\int_{\re^{n}}
			|\N K(x)|
			\,
			\dx
			+
			3^{n}
			\| 
			\N K
			\|_{r}.
			\\
		}
	}
	Combining relations \eqref{eq;bessel-gradient-estimate1}-\eqref{eq;bessel-gradient-estimate5}, 
	we conclude 
	$$
	\| \N K * f \|_{L^{p}(Q_{\frac12}(k))}
	\le
	3^{n}
	(5^{n}\| \N K \|_{1} + 3^{n} \| \N K \|_{r})
	\sup_{k' \in \mathbb{Z}^{n}}
	\| f \|_{L^{q}(Q_{\frac12}(k'))}
	$$
	and thus
	$$
	\sup_{k \in \mathbb{Z}^{n}}
	\| \N K * f \|_{L^{p}(Q_{\frac12}(k))}
	\le
	(
	15^{n}
	\| \N K \|_{1}
	+
	9^{n}
	\| \N K \|_{r}
	)
	\sup_{k' \in \mathbb{Z}^{n}}
	\| f \|_{L^{q}(Q_{\frac12}(k'))},
	$$
	where 
	$1+\frac1p=\frac1q+\frac1r$ for $1 \le r<\frac{n}{n-1}$.
	Thanks to Lemma~\ref{lem;ball-qube-equiv},
	we obtain
	$$
	\|
	\N K
	*
	f
	\|_{\uloc{p}}
	\le
	(
	90^{n}
	\| \N K \|_{1}
	+
	54^{n}
	\| \N K \|_{r}
	)
	\| f \|_{\uloc{q}},
	$$
	which completes the proof. 
\end{proof}

\begin{lemma}\label{lem;Duhamel-term-est}
	For each $p$ satisfying condition \eqref{eqn:warP}  there exists $k\in (n,\infty]$ and a number $C = C(p,k,n,\nabla K)$ such that 
	\begin{equation}\label{eq;Duhamel-term-est}
		\begin{split}
			\| 
			\N
			e^{\tau\Delta} & \cdot (
			u
			\N K*v )
			\|_{\uloc{p}} \le 
			C\tau^{-\frac12}
			(
			1
			+
			\tau^{-\frac{n}{2}\frac{1}{k}}
			)
			\|
			u
			\|_{\uloc{p}}
			\|
			v
			\|_{\uloc{p}}
		\end{split}  
	\end{equation}
	for all $u,v \in \Lpu (\R^n)$ and  $\tau >0$.
\end{lemma}

\begin{proof}
	Using the heat semigroup estimates from Proposition \ref{prop;heat-semi-uloc} and the H\"older inequality (which also holds in $\Lpu$-norm) we have 
	\eqn{
		\spl{
			\| 
			\N
			e^{\tau\Delta} \cdot
			( u
			\N K*v )
			\|_{\uloc{p}}
			\le& 
			C \tau^{-\frac12}
			(
			1
			+
			\tau^{-\frac{n}2  \frac1{k} }
			)
			\|
			u
			\N K*v
			\|_{\uloc{r}}\\
			\le&
			C\tau^{-\frac12}
			(
			1
			+
			\tau^{-\frac{n}2  \frac1{k}}
			)
			\|
			u
			\|_{\uloc{p}}
			\|
			\N K*v
			\|_{\uloc{k}},
		}
	}
	where
	$$\frac{1}{r} = \frac{1}{p}+\frac{1}{k} \quad \text{and} \quad r \ge 1.$$
	Moreover, applying Proposition \ref{lem;bessel-gradient-estimate}, we obtain 
	\eqspn
	{	
		\|
		\N K*v
		\|_{\uloc{k}} \le C (\| \nabla K\|_1 + \| \nabla K \|_{q_1}) \|v\|_{\uloc{p}}  
	}
	with
	$$
	1 + \frac{1}{k} 
	= 
	\frac{1}{p} + \frac{1}{q_1} 
	\quad 
	\text{and} 
	\quad 
	q_1 \in \Big[ 1, \frac{n}{n-1} \Big). $$
	Let us show that for every $p$ satisfying condition~\eqref{eqn:warP} we can always choose $k \in (n,\infty]$ and $r \ge 1$ satisfying the conditions above. Indeed,  if $n >1$ then for every $k\in [p,2p]$ we have 
	\eqspn{  
		\frac{1}{q_1} 
		= 
		1 + \frac{1}{k}- \frac{1}{p} 
		\le 
		1 
		\quad 
		\text{and} 
		\quad 
		\frac{1}{q_1} 
		= 
		1 + \frac{1}{k}- \frac{1}{p} 
		\ge 
		1 - \frac{1}{2p} 
		> 1 - \frac{1}{n}.
	}
	Analogously if $n=1$ then the inequalities holds true for every $k \in [p,\infty]$. 
	Next, the condition $r\ge 1$ is equivalent to $k\in [p/(p-1),\infty ]$. 
	For $p$ satisfying condition~\eqref{eqn:warP} and $n >1$ the intersection
	${[p, 2p] \cap [p/(p-1),\infty ] \cap (n, \infty]}$
	is nonempty and hence the choice of $k$ is always possible. Analogously, if $n=1$ we choose arbitrary $k \in [ p/(p-1),\infty ] \cap (n, \infty]$.
\end{proof}

We obtain a solution to integral equation \eqref{eq;IE} from the Banach fixed point theorem formulated in the following way.

\begin{proposition}\label{prop;Banach-fixed-pt}
	Let $X$ be a Banach space and let 
	$Q=Q[\cdot , \cdot] : X \times X \to X$ be a bounded bilinear form 
	with
	$$
	\| 
	Q[u,v]
	\|_{X}
	\le
	C_{1}
	\|
	u
	\|_{X}
	\|
	v
	\|_{X}
	$$
	for some $C_{1}>0$ independent of $u,v \in X$.
	Assume that 
	$\delta \in (0,1/({4C_{1}}))$.
	If 
	$
	\| 
	y_{0} 
	\|_{X}
	\le 
	\delta
	$,
	then the equation 
	$$
	u
	=
	y_{0}
	+
	Q[u,u]
	$$
	has a solution with $\| u \|_{X} \le 2\delta$.
	This solution is unique in the set
	$
	\big\{ u \in X \, : \, \| u \|_{X} \le 2\delta \big\}
	$
	and stable in the following sense:
	if $y_{0}, \widetilde{y}_{0} \in X$ satisfy $\| {y}_{0}\|_{X} \le \delta$, $ \|\widetilde{y}_{0}\|_{X} \le \delta$
	then for the corresponding solutions $u,\widetilde{u} \in X$ we have
	$$
	\|
	u-\widetilde{u}
	\|_{X}
	\le
	C_2
	\|
	y_{0}-\widetilde{y}_{0}
	\|_{X},
	$$
	where $C_2>0$ is independent of $u$ and $\widetilde{u}$.
\end{proposition}

\begin{proof}[Proof of Theorem \ref{thm;main}]
	For $T>0$, we introduce 
	$ X_T \equiv L^{\infty}\big((0,T);\Lpu(\re^{n})\big) $
	which is a Banach space with the norm $	\| u \|_{X_T} \equiv \sup_{t\in (0,\, T)} \| u(t) \|_{\uloc{p}}.$
	In order to apply Proposition~\ref{prop;Banach-fixed-pt}, it suffices to estimate the bilinear form
	$$
	Q[u,v](t)
	=
	- \int_{0}^{t}
	\N
	e^{(t-s)\Del} \cdot
	\big(u(s) \N K* v(s)\big)
	\, \ds .
	$$
	By Lemma~\ref{lem;Duhamel-term-est}, 
	for some $k>n$,  we obtain 
	\eq{\label{eq;IneqEst}
		\spl{
			\|Q[u,v](t)\|_{\uloc{p}} 
			& \le  
			\int_{0}^{t} 
			\| 
			\N e^{(t-s)\Delta} \cdot (u(s) \N K*v(s))  
			\|_{\uloc{p}}\\
			& \le
			C\int_{0}^{t}
			(t-s)^{-\frac12}
			(
			1
			+
			(t-s)^{-\frac{n}{2}\frac1{k}}
			)
			\|
			u(s)
			\|_{\uloc{p}}
			\|
			v(s)
			\|_{\uloc{p}}
			\,
			\ds\\
			& \le
			C
			\|
			u
			\|_{X_T}
			\|
			v
			\|_{X_T}
			\int_{0}^{t}
			(t-s)^{-\frac12}
			\big(
			1
			+
			(t-s)^{-\frac{n}{2}\frac1{k}}
			\big)
			\,
			\ds 
			\\
			& \le
			C
			\big(t^{\frac12}
			+
			t^{\frac12-\frac{n}{2 k}}
			\big)
			\|
			u
			\|_{X_T}
			\|
			v
			\|_{X_T}.
		}
	}
	Therefore, we have
	\eq{\label{eq;Duhamel-est-smallness}
		\|
		Q[u,v]
		\|_{X_T}
		\le
		C
		\big(
		T^{\frac12}
		+
		T^{\frac12-\frac{n}{2 k}}
		\big)
		\|
		u
		\|_{X_T}
		\|
		v
		\|_{X_T}.
	}
	Since 
	inequality \eqref{eq;heat-semi-Lp-Lp}
	in Proposition \ref{prop;heat-semi-uloc} provides 
	$
		\|e^{t\Delta} u_0 \|_{X_T} \le \|u_0\|_{\uloc{p}}
	$, 
	we obtain a solution to the integral equation via  Proposition \ref{prop;Banach-fixed-pt} for sufficiently small $T>0$. 
	 In order to show that $u \in C \bigl( (0,T); \Lpu(\re^n) \bigr)$ it suffices to follow the arguments from \cite[p.~388]{Ma-Te}.
	This solution is unique  by a usual reasoning.

	To show that a solution is non-negative in the case of non-negative initial datum, 
	we pass through an approximation process with a sequence of smooth solutions $\{ u^{\ep} \}_{\ep > 0}$ corresponding to the smooth, uniformly  
	bounded non-negative initial conditions
	$
	u^{\ep}_{0}(x)
	=
	(e^{\ep \Delta}u_{0})(x).
	$ 
	Here, $ u^{\ep} \ge 0$ by the classical  maximum principle, see {\it e.g.}~\cite[Theorem 9, p.43]{fred}. 
	To complete this reasoning, 
	we show that $\|u^\varepsilon (t) - u(t)\|_{\uloc{p}} \to 0$ as $\varepsilon \to 0$ for each $t \in (0,T)$. 
	Indeed, 
	from inequality \eqref{eq;heat-semi-Lp-Lp} in Proposition \ref{prop;heat-semi-uloc}, 
	it holds that 
		$
			{\|u_0^\varepsilon \|_{\uloc{p}}
			\le 
			\|u_0 \|_{\uloc{p}}}
		$
	for all $\varepsilon > 0$.
	Hence, by the above construction of a local-in-time solutions via Proposition~\ref{prop;Banach-fixed-pt}, there exists a constant $ M>0$ such that $\|u\|_{X_T} \le M$ and $\|u^\varepsilon\|_{X_T} \le M$.
	Now, computing  the norm $\|u^\varepsilon(t) - u(t)\|_{\uloc{p}}$ and using the integral representation \eqref{eq;IE} of these functions as well as a second inequality in  estimate \eqref{eq;IneqEst} we obtain
	\eqspn{
		&\| u^\varepsilon (t) - u(t) \|_{\uloc{p}} \\
		\le&
		\|e^{t\Delta} u_0^\varepsilon - e^{t\Delta} u_0 \|_{\uloc{p}}\\
		&+ 
		C 
		\big( 
		\|u^\varepsilon\|_{X_T} + \|u\|_{X_T} 
		\big) 
		\int_0^t 
		(t-s)^{-\frac12}
		\big(
		1 +(t-s)^{-\frac{n}{2}\frac1{k}} 
		\big) 
		\| u^\varepsilon(s) - u(s) \|_{\uloc{p}}
		\, {\rm d}s.
	}
	Applying the Volterra  type inequality from \cite[Ch.~9]{yagi} 
	we conclude that
	\eqsp{\label{eq:PosConv} 
		\|u^\varepsilon (t) - u(t) \|_{\uloc{p}} 
		\le&
		\|e^{t\Delta} u_0^\varepsilon - e^{t\Delta} u_0 \|_{\uloc{p}} 
		+ 
		C 
		\int_0^t 
		\kappa(t-s) 
		\|e^{s\Delta} u_0^\varepsilon - e^{s\Delta} u_0 \|_{\uloc{p}} 
		\, \ds,
	}
	where $\kappa = \kappa(\tau)$ is a suitable integrable kernel on $[0,t]$. Finally, by Remark \ref{prop;uloc-equivalent-cond-closure} and the semigroup properties, we have 
	\eqspn{
		\| e^{t\Delta} u_0^\varepsilon - e^{t\Delta} u_0 \|_{\uloc{p}} 
		= 
		\| e^{\varepsilon \Delta} e^{t\Delta} u_0 - e^{t\Delta} u_0 \|_{\uloc{p}} 
		\to 
		0 
		\quad 
		\text{as} 
		\quad 
		\varepsilon \to 0 .
	}
\end{proof}

\section{Linearized problem}
\label{sec:OpPr}

\subsection{Preliminary properties}
The linearization procedure described below in Section~\ref{sec:GlobStab} leads to  the following linear problem
\eq{\label{eq;linearized-linear-DD}
	\left\{
	\spl{
		&v_{t}- \Del v + A\Del K \ast v = 0,\quad t>0, \ \ x\in \re^n\\
		&v(0,x)=v_{0}(x),
	}
	\right.
}
where 
$A \in \re$ is an arbitrary constant and 
the operator $\OPL$  can be expressed by the Fourier transform as follows
$$
(
\Del \varphi - A \Del K \ast \varphi
)
\widehat{\enspace}(\xi) 
= 
\bigg(
-|\xi|^2 + A \frac{|\xi|^2}{1+|\xi|^2} 
\bigg) 
\widehat{\varphi}(\xi), 
\quad \xi \in \re^n.
$$
We begin by presenting preliminary properties of this operator.

\begin{lemma}
	\label{lem;finiteness-kernel}
There exists a constant $L>0$ such that 
for each $p\in [1, \infty]$,
	\begin{equation}\label{eq;stein}
		\| -\Del K \ast v \|_{p} 
		\le 
		L \| v \|_{p} 
		\qquad \text{for all} \quad v \in L^p(\re^n).
	\end{equation}
\end{lemma}
This lemma is an immediate consequence  of the fact that the operator $-\Del K \! \ast $ can be represented by a convolution with a finite measure on $\re^{n}$.
We skip the proof of this classical result from the harmonic analysis, see \textit{e.g.}~\cite[Lemma~2.(i), p.133]{Stein}.

\begin{lemma}
	For each $A \in \re$, 
	a closure in $L^p(\R^n)$ of the operator $\OPL$ generates an analytic semigroup 
	$\left\lbrace S_A(t) \right\rbrace_{t \ge 0}$ 
	on 
	$L^p(\re^n)$ 
	for every $p \in [1,\infty)$. 
	This semigroup is defined by the Fourier transform by the formula 
	\begin{equation}\label{eq;semi-def}
		\widehat{S_A(t)v_0}(\xi) 
		= 
		\widehat{\mu}_A(t, \xi) \widehat{v_0}(\xi),
	\end{equation}
	where
	\begin{equation}\label{eq;mu-def}
		\widehat{\mu}_A(t, \xi) 
		= 
		e^{-t \big( |\xi|^2 - A  \frac{|\xi|^2}{1+|\xi|^2} \big) }
		{\rdd .}
	\end{equation}
\end{lemma}
\begin{proof}
	It is well-known that Laplacian generates an analytic semigroup of linear operators on $L^p(\R^n)$ for every $p\in [1, \, \infty)$. 
	A bounded perturbation of such an operator maintains the same property, see 	\textit{e.g.} \cite[Chapter I\hspace{-.1em}I\hspace{-.1em}I, Theorem 2.10]{EN00}.
	The Fourier representation of this semigroup can be obtained by routine calculations. 
\end{proof}

\begin{lemma}\label{lem;linearized-linear-DD-2}
	Assume that $A \in \R$ and choose the constant $L$ from inequality~\eqref{eq;stein}. 
	Then for each $1\le q \le p \le \infty$, there exists a constant $C = C(p,q,n) > 0$ such that
	\eqspn{\label{ineq;11738}
		\|S_A(t) v_0\|_p 
		\le 
		C 
		t^{-\frac{n}{2}( \frac{1}{q} - \frac{1}{p} )}
		e^{|A|Lt} 
		\| v_0 \|_q}
	and
	\eqspn{
		\|\nabla S_A(t) v_0\|_p 
		\le 
		C 
		t^{-\frac{n}{2}( \frac{1}{q} - \frac{1}{p}) -\frac{1}{2}}
		e^{|A|Lt} \| v_0 \|_q
	}
	for all $t>0$ and $v_0 \in L^q(\R^n)$.
\end{lemma}
\begin{proof}
	Here, we use the notation
	$S_A(t) v_0 = e^{t \Del} ( e^{-A\Del K \! \ast} v_0 )$.
	Using the $L^p$-$L^q$ estimates of the heat semigroup (see Remark~\ref{rem:heatest})
	and
	Lemma~\ref{lem;finiteness-kernel}, 
	we obtain
	$$ 
	\|
	e^{t\Delta} (e^{-t\OPs} v_0 ) 
	\|_p 
	\le 
	C 
	t^{-\frac{n}{2}( \frac{1}{q} - \frac{1}{p}) }  
	\|
	e^{-t\OPs} v_0
	\|_q 
	\le 
	C 
	t^{-\frac{n}{2}( \frac{1}{q} - \frac{1}{p}) } e^{|A|Lt} \|v_0\|_q. 
	$$
	The proof for the second inequality is analogous.
\end{proof}

\subsection{Decay estimates when $A<1$}
The following theorem improves estimates from Lemma~\ref{lem;linearized-linear-DD-2} in the case of $A \in [0,1)$ and it plays a crucial role in the proof of stability of constant solutions to problem \eqref{eq;DD}.

\begin{theorem}\label{lem;Lp-Lq-estimate-S}
	Assume that $A \in [0,1)$. For all exponents satisfying $1 \le q \le p \le \infty$ there exist constants $C = C(p,q,n,A) > 0$ such that
	\begin{equation}\label{eqn;Stpq}
		\| S_A(t) v_{0} \|_{p} 
		\le 
		C 
		t^{-\frac{n}{2} ( \frac1q - \frac1p )} 
		\| v_{0} \|_{q}
	\end{equation}
	and
	\begin{equation}\label{eqn;dStpq}
		\| \N S_A(t) v_{0} \|_{p} 
		\le 
		C 
		t^{- \frac{n}{2} ( \frac1q - \frac1p ) - \frac{1}{2} } 
		\| v_{0} \|_{q}
	\end{equation}
	for all $t > 0$ and $v_{0} \in L^{q}(\re^{n})$.
\end{theorem}

The proof of this theorem is based on the following lemmas.

\begin{lemma}\label{lem:ineq-bh}
	Let $\widehat{D^N v} (\xi) \equiv |\xi|^N \widehat{v} (\xi)$ for all $N \in \re$. 
	For all $v \in \mathcal{S}(\re^n)$ and for every $N>\frac{n}{2}$, 
	the following inequality holds
	$$
	\| v \|_1 
	\le 
	C 
	\| \widehat{v} \|^{1-\frac{n}{2N}}_2 
	\|D^N\widehat{v}\|^{\frac{n}{2N}}_2,
	$$
	with a constant $C = C(n,N)>0$.
\end{lemma}
\begin{proof}
	For 
	$
	R 
	=
	\big( 
	\frac{2N-n}n 
	\frac{\|D^N\widehat{v}\|_2}{{\| \widehat{v} \|_2}} 
	\big)^{1/N}
	$, 
	we obtain
	\begin{equation*}
		\begin{split}
			\| v \|_1 
			= & 
			\int_{|x| \le R} |v(x)| \, \d{x} + \int_{|x| > R} |v(x)| \, \d{x} \\
			\le & 
			\Big( 
			\int_{|x| \le R} \, \d{x} 
			\Big)^\frac{1}{2} 
			\Big( 
			\int_{|x| \le R} |v(x)|^2 \, \d{x} 
			\Big)^\frac{1}{2} 
			 + 
			\Big( 
			\int_{|x| > R} |x|^{-2N} \, \d{x} 
			\Big)^\frac{1}{2} 
			\Big( 
			\int_{|x| > R} |x|^{2N} |v(x)|^2 \, \d{x} 
			\Big)^\frac{1}{2} \\
			\le & 
			|\mathbb{S}^{n-1}|^{\frac12} 
			(
			R^{\frac{n}{2}} \|v\|_2 
			+ 
			R^{\frac{n}{2}-N} \|D^N \widehat{v}\|_2
			) \\
			\le & 
			C 
			\|\widehat{v}\|^{1-\frac{n}{2N}}_2 
			\|D^N\widehat{v}\|^{\frac{n}{2N}}_2.
		\end{split}
	\end{equation*}
\end{proof}

\begin{lemma}\label{lem;DNmu-estim}
	Assume that $A \in [0,1)$. 
	For the function $\widehat{\mu}_A$ defined by formula~\eqref{eq;mu-def} and for every multi-index $\alpha$ with $|\alpha| = N$, 
	there exists a constant 
	$C=C(n,A,N) >0$ 
	such that
	$$
	\|\partial^{\alpha}_{\xi} \widehat{\mu}_A(t)\|^2_2 
	\le 
	C t^{N - \frac{n}{2}}, 
	\quad \text{for all} \quad t \ge 1.
	$$
\end{lemma}
\begin{proof}
	For $N=0$, 
	by the inequality $|\xi|^2/(1+|\xi|^2) \le |\xi|^2 \ (\xi \in \re^{n})$, 
	we obtain
	$$
	\|\widehat{\mu}_A(t)\|^2_2 
	= 
	\int_{\re^n} e^{- 2t |\xi|^2 + 2A t \frac{|\xi|^2}{1+|\xi|^2}} \, \d{\xi} 
	\le
	\int_{\re^n} e^{-2 t (1-A) |\xi|^2} \, \d{\xi} 
	= 
	Ct^{-\frac{n}{2}}.
	$$
	For $N\ge1$, 
	we introduce the $C^{\infty}$-function 
	$
	h(\xi) \equiv |\xi|^2 - A |\xi|^2 / (1+|\xi|^2)
	$ 
	which satisfies estimates 
	$ 
	|\partial_{\xi_j} h(\xi)| \le C |\xi|
	$ 
	and 
	$
	|\partial^{\beta}_{\xi} h(\xi)| \le C
	$ 
	for every $j$, 
	$1 \le j \le n$ and multi-index $\beta$ with $|\beta| \ge 2$. 
	We use the multivariate Fa\`{a} di Bruno's formula (see \textit{e.g.} \cite{Hardy2006CombinatoricsOP})
	$$
	\partial^{\alpha}_{\xi} e^{- t h(\xi)} 
	=
	e^{- t h(\xi)} 
	\sum_{k=1}^{N} (-t)^{k} H_k(\xi),
	$$
	where $H_k(\xi)$ is a sum of products of $k$ partial derivatives of the function $h(\xi)$ such that $| H_k(\xi) | \le C(1+|\xi|^k)$.
	We prove the following inequality by induction in $N \in \Nt$,
	$$
	\big| \partial^{\alpha}_{\xi} e^{-th(\xi)} \big| 
	\le 
	C e^{-th(\xi)} \sum_{k - \frac{\ell}{2} 
		\le 
		\frac{N}{2}} t^k \big(1 + |\xi|^{\ell} \big).
	$$
	For $N=1$, 
	the inequality is obvious. We show the induction step for $N+1$,
	\begin{equation*}
		\begin{split}
			\big| 
			\partial_{\xi_j} \partial^{\alpha}_{\xi} e^{-th(\xi)} 
			\big| 
			&\le 
			t 
			\big|
			\partial_{\xi_j} h(\xi) 
			\big| 
			\big| 
			\partial^{\alpha}_{\xi} e^{-th(\xi)} 
			\big| 
			+ 
			\bigg| 
			e^{- t h(\xi)} 
			\sum_{k=1}^{N} \,
			(-t)^{k} \partial_{\xi_j} H_k(\xi) 
			\bigg| \\
			& \le 
			C 
			t 
			|\xi| 
			e^{-th(\xi)} 
			\sum_{k - \frac{\ell}{2} \le \frac{N}{2}} 
			t^k \big(1 + |\xi|^{\ell} \big) 
			+ 
			C
			e^{-th(\xi)} 
			\sum_{k - \frac{\ell}{2} \le \frac{N}{2}} 
			t^k 
			\big(1 + |\xi|^{\ell} \big) \\
			&
			\le 
			C e^{-th(\xi)} \! 
			\sum_{k - \frac{\ell}{2} \le \frac{N+1}{2}} 
			t^k \big(1 + |\xi|^{\ell} \big)
		\end{split}
	\end{equation*}
	which holds true because 
	$
	|\partial_{\xi_j} H_k(\xi)| \le C(1 + |\xi|^{\ell} )
	$ 
	by the properties of the function $h(\xi)$ and $k + 1 - (\ell+1)/2 \le (N+1)/2$. 
	
	Now we group coefficients $t^k$ and $|\xi|^\ell$ in the following way, 
	$
	t^k |\xi|^\ell 
	= 
	t^{k-\ell/2} |\sqrt{t} \xi|^\ell
	$ 
	thus, 
	by the assumption $t \ge 1$ and induction, 
	we have 
	$
	t^{k-\ell/2} \le t^{N/2}
	$. 
	We obtain an estimate
	$$
	|\partial^{\alpha}_{\xi} \widehat{\mu}_A(t)| 
	\le 
	C 
	t^{\frac{N}{2}} 
	P(|\sqrt{t} \xi|) 
	e^{- t |\xi|^2 + A t \frac{|\xi|^2}{1+|\xi|^2}},
	$$
	where $P(s)$ is a polynomial of degree $N$. 
	By the same inequality as in the case $N=0$ and properties of the exponential function,
	\begin{equation}\label{eq;mua_part}
		|\partial^{\alpha}_{\xi} \widehat{\mu}_A(t)| 
		\le 
		t^{\frac{N}{2}} 
		P(|\sqrt{t} \xi|) 
		e^{-t (1-A) |\xi|^2} 
		\le 
		C 
		t^{\frac{N}{2}} 
		e^{- \delta t |\xi|^2},
	\end{equation}
	for some $\delta \in (0,1-A)$. 
	Calculating the $L^2$-norm of both sides of inequality~\eqref{eq;mua_part} we obtain the result.
\end{proof}

\begin{lemma}\label{lem;St-bound}
	Assume that $A \in [0,1)$. 
	For every $p \in [1, \infty]$ there exists a constant $C>0$ such that
	$$ 
	\|S_A(t) v_0\|_p \le C \|v_0\|_p,
	$$
	for all $t>0$ and all $v_0 \in L^p(\re^n)$.
\end{lemma}
\begin{proof}
	For $t \in [0,1]$, 
	this is an immediate consequence of Lemma \ref{lem;linearized-linear-DD-2}. 
	For $t \ge 1$, 
	the function $\mu_A$ is from the Schwartz class in variable $\xi$. 
	Thus, 
	by the Young inequality,
	$$ 
	\|S_A(t) v_0\|_p 
	= 
	\| \mu_A(t) \ast v_0\|_p 
	\le 
	\|\mu_A(t)\|_1
	\|v_0\|_p.
	$$
	In order to estimate $\|\mu_A(t)\|_1$, 
	we recall a well known fact that both quantities $\|D^N v\|_2$ 
	and 
	$
	\sum_{|\alpha| = N} \| \partial_x^{\alpha} v \|_2
	$
	are comparable for each $N \in \Nt$. 
	Combining Lemma \ref{lem:ineq-bh} and Lemma~\ref{lem;DNmu-estim}, 
	for $N > n/2$ and $N \in \Nt$, 
	we obtain
	\begin{equation*}
		\begin{split}
			\|\mu_A(t)\|_1 
			\le 
			C
			\|\widehat{\mu}_A(t)\|^{1-\frac{n}{2N}}_2 
			\|D^N\widehat{\mu}_A(t)\|^{\frac{n}{2N}}_2 
			\le 
			C
			{( C_1 t^{- \frac{n}{4}})}^{1-\frac{n}{2N}}
			{(C_2 t^{\frac{N}{2}- \frac{n}{4}})}^{\frac{n}{2N}} 
			= 
			C
		\end{split}
	\end{equation*}
	for all $t \ge 0$.
\end{proof}

\begin{proof}[Proof of Theorem~\ref{lem;Lp-Lq-estimate-S}]
	We begin with inequality \eqref{eqn;Stpq}. 
	Let us choose $\varepsilon \in (A,1)$. 
	By the standard heat semigroup estimates (see Remark \ref{rem:heatest}),
	\eqsp{
	\| S_A(t) v_0 \|_p 
	= 
	\| 
	e^{(1-\varepsilon)t \Delta} 
	( e^{\varepsilon t \Delta - t A \Delta(I -\Delta)^{-1}} v_0 ) 
	\|_p 
	\le 
	C_1 
	t^{-\frac{n}{2} ( \frac1q - \frac1p )} 
	\|
	e^{\varepsilon t \Delta - t A \Delta(I -\Delta)^{-1}} v_0
	\|_q.
	}
	Now we substitute $\tilde{t} = \varepsilon t$ to obtain 
	$
	tA 
	= 
	\tilde{t} (A/\varepsilon) 
	= 
	\tilde{t} \tilde{A}
	$ 
	and 
	$
	0 \le \tilde{A} < 1
	$. 
	Thus, 
	by Lemma~\ref{lem;St-bound},
	$$
	\| S_A(t) v_0 \|_p 
	\le 
	C_1 
	t^{-\frac{n}{2} ( \frac1q - \frac1p )} 
	\|
	e^{\tilde{t} \Delta - \tilde{t} \tilde{A} \Delta(I -\Delta)^{-1}} v_0
	\|_q 
	\le 
	C 
	t^{-\frac{n}{2} ( \frac1q - \frac1p )} 
	\|v_0\|_q.
	$$
	We prove inequality \eqref{eqn;dStpq} analogously using the formula
	$$
	\N S_A(t) v_{0} 
	= 
	\N e^{(1-\varepsilon)t \Delta} 
	( 
	e^{\varepsilon t \Delta - t A \Delta(I -\Delta)^{-1}} v_0 
	).
	$$
\end{proof}

\begin{remark}\label{rem;Abelow0}
	The $L^q$-$L^p$ estimates \eqref{eqn;Stpq}-\eqref{eqn;dStpq} of the semigroup 
	$\big\{ S_A(t) \big\}_{t \ge 0}$ 
	hold true for $A<0$ as well. 
	They can be proved by the same reasoning as above using the obvious inequality
	$$
	e^{-t \big( |\xi|^2 - A  \frac{|\xi|^2}{1+|\xi|^2} \big)} 
	\le 
	e^{-t |\xi|^2} 
	\quad \text{for each } A<0.
	$$
\end{remark}

\begin{remark}\label{rem;stabA1}
	For the completeness of this work, we notice that constant solution $A = 1$ is linearly stable in $L^2(\re^n)$. Indeed, since
	$ e^{-t ( |\xi|^2 - {|\xi|^2}/({1+|\xi|^2}))} \le 1$
	for all $\xi \in \re^n$ and $t \ge 0$, by the Plancherel formula, we obtain
	$ \| S_1 (t) v_0 \|_2 \le \|v_0\|_2$
	for all $v_0 \in L^2(\re^n)$. 
	We skip a discussion  of a stability of this constant solution for $p \neq 2$.
\end{remark}

\subsection{Exponential growth when $A>1$}
Next, 
we study an instability of solutions to linear problem \eqref{eq;linearized-linear-DD}.

\begin{lemma}\label{lem;spectrum}
	Let $p \in (1,\infty)$. 
	The closure in $L^p(\R^n)$ of the operator $\OPL$ has a~real continuous spectrum 
	$(-\infty, a]$
	,
	where $a = 0$ 
	if $A \le 1$ 
	and 
	$a = (\sqrt{A} - 1)^2$ if $A > 1$.
\end{lemma}
\begin{proof}
	The symbol $h(\xi) = |\xi|^2 - A |\xi|^2 / (1+|\xi|^2)$ of the operator satisfies the estimates
	$$
	|\partial^{\alpha}_\xi h(\xi)| 
	\le 
	C_{\alpha} 
	(1+|\xi|)^{2-|\alpha|} 
	\quad 
	\text{for all } \xi \in \re^n
	$$
	and each multi-index $\alpha = (\alpha_1, ..., \alpha_n)$. 
	Moreover $h(\xi)^{-1} = O(|\xi|^{-2})$ as $|\xi| \to \infty$.
	Now, 
	it suffices to apply the result from \cite[Theorem 2.1]{Wong}.
\end{proof}

\begin{lemma}\label{lem;upperA}
	Assume $A > 1$, $a = (\sqrt{A} - 1)^2$ and $1 < q \le p < \infty$. 
	For every $\varepsilon > 0$ there exists a constant $C = C(\varepsilon, p, q, A) > 0$ 
	such that
	\begin{equation}\label{eq:upperA1}
		\|S_A(t) v_0 \|_p 
		\le 
		C t^{-\frac{n}{2} ( \frac1q - \frac1p )}
		e^{(a+\varepsilon) t} 
		\|v_0\|_q
	\end{equation}
	and
	\begin{equation}\label{eq:upperA2}
		\|\N S_A(t) v_0 \|_p 
		\le 
		C
		t^{-\frac{n}{2} ( \frac1q - \frac1p ) - \frac{1}{2}} 
		e^{(a+\varepsilon) t} 
		\|v_0\|_q
	\end{equation}
	for all $v_0 \in L^q(\re^n)$ and $t \ge 0$.
\end{lemma}
\begin{proof}
	Estimate \eqref{eq:upperA1} for $p = q$ is a direct consequence of Lemma \ref{lem;spectrum} combined with estimates of strongly continuous semigroups, 
	see \textit{e.g.} \cite[Ch.I\hspace{-.1em}V, Corollary 3.12]{EN00}. 
	In order to prove inequality \eqref{eq:upperA2} for $p=q$, 
	notice that for every $\delta \in (0,1)$ we have
	$$
	\N S_A(t) v_0 
	= 
	\N 
	\big( 
	e^{\delta t \Del} 
	( 
	e^{t((1-\delta)\Del - A \Del K \! \ast)} 
	) 
	\big) 
	v_0 
	= 
	\N e^{\delta t \Del} 
	\big( 
	e^{t(1-\delta)\left(\Del - \frac{A}{1-\delta} \Del K \! \ast\right)} 
	v_0
	\big)  
	.
	$$
	Applying the estimates of the heat semigroup from Remark~\ref{rem:heatest} and inequality~\eqref{eq:upperA1} with $\varepsilon/2$ to the operator $S_\frac{A}{1-\delta}(t)$ yields
	\eqspn{
		\|\N S_A(t) v_0\|_p
		\le 
		C(\varepsilon/2) 
		(\delta t)^{-1/2} 
		e^{t(1-\delta)\left(a_\delta +\varepsilon/2 \right)} 
		\| v_{0} \|_{p}
		\quad 
		\text{with} 
		\quad  
		a_\delta 
		= 
		\Big( \sqrt{\frac{A}{1-\delta}} - 1\Big)^2.
	}
	Choosing $\delta > 0$ sufficiently small to have 
	$
	a_\delta \le a + \varepsilon/2
	$. 
	
	For $q < p$ we proceed analogously using $L^q$-$L^p$ estimates of the heat semigroup from Remark~\ref{rem:heatest}.
\end{proof}

\begin{lemma}
	\label{lem:SMis}
	Assume $A > 1$, $a = (\sqrt{A} - 1)^2$ and $p \in (1, \infty)$. 
	For every $\gamma \in (0,\,1]$ and every $T>0$, 
	there exists $v_0\in L^p(\re^n)$ such that  for each $t\in [0,T]$
	\eqsp{
		\| 
		S_A(t)v_0 - e^{at} v_0 
		\|_p 
		\le 
		\gamma 
		\|v_0\|_p
		\label{eq:VEs} 
		\quad \text{\rm and} \quad 
		\| 
		S_A(t) v_0 
		\|_p 
		\le 
		2 
		e^{ a t} 
		\|v_0\|_p.
	}
\end{lemma}
\begin{proof}
	It follows from Lemma \ref{lem;spectrum} that $a = (\sqrt{A} - 1)^2$ lies on the boundary of the spectrum of the operator $\OPL$ and such elements belong to the approximate point spectrum (see \textit{e.g.} \cite[Lemma~1]{ShSt}). Thus, for each $\varepsilon >0$ there exists $v_\varepsilon \in L^p(\R^n)$ (in fact, $v_\varepsilon$ belongs to the domain of the closure in $L^p(\re^n)$ of the operator $\OPL$) such that
	\eqsp{\label{eq;AppEigVec}
		\| \Delta v_\varepsilon - A\Delta K*v_\varepsilon - a v_\varepsilon \|_p 
		\le 
		\varepsilon 
		\|v_\varepsilon\|_p.
	}
	Hence, 
	by usual calculations for semigroups of linear operators (see \textit{e.g.} \cite[Ch.I\hspace{-.1em}I, Sec.~3]{EN00}) 
	and by Lemma \ref{lem;linearized-linear-DD-2}, we obtain
	\eqspn{
		\| 
		S_A(t) v_\varepsilon - e^{at} v_\varepsilon 
		\|_p 
		&= 
		\bigg\| 
		\int_0^1 
		\frac{{\rm d}}{{\rm d}s}e^{ts(\OPL)}e^{t(1-s)a}  v_\varepsilon
		\, \ds 
		\bigg\|_p \\
		&\le 
		\int_0^1 
		\big\| 
		e^{ts(\OPL)}e^{t(1-s)a} t(\Delta v_\varepsilon - A\Delta K*v_\varepsilon - a v_\varepsilon)  
		\big\|_p 
		\, \ds \\
		& \le 
		C
		t 
		\| 
		\Delta v_\varepsilon - A\Delta K*v_\varepsilon - a v_\varepsilon 
		\|_p 
		\int_0^1 
		e^{|A|Lts} e^{t(1-s)a} 
		\, \ds,
	}
	for some constant $C>0$ from inequality \eqref{ineq;11738} with $p=q$. 
	Therefore, inequality~\eqref{eq;AppEigVec} with $\varepsilon = \frac{\gamma}{CTe^{(|A|L+a)T}}$ provides the estimate
	\eqspn{
		\| 
		S_A(t) v_0 - e^{at} v_0 
		\|_p 
		\le 
		\frac{te^{(|A|L + a) t}}{Te^{(|A|L + a) T}} 
		\gamma 
		\|v_0 \|_p \le \gamma \|v_0\|_p.
	}
	
	Since $a>0$, the second inequality in \eqref{eq:VEs} can be obtained immediately from the first one by choosing $\gamma = 1$.
\end{proof}


\section{Perturbations of constant solutions}
\label{sec:GlobStab}

We study a solution $u = u(t,x)$ of problem \eqref{eq;DD} which is a perturbation of the constant stationary solution $A \in \R$. Thus, the function $v(t,x) = u(t,x) - A$ satisfies
\eq{\label{eq;linearized-DD}
	\left\{
	\spl{
		& v_t - \Del v + A\Del K \ast v + \N \cdot (v \N K*v) = 0,&&\quad t > 0,\ \  x \in \re^{n},\\
		&v(0,x)=v_{0}(x),&&\quad x \in \re^{n}.
	}
	\right.
}
Here, in fact, we consider a mild solution to this problem satisfying the integral equation
\eqsp{
	v(t) 
	= 
	S_A(t) v_0 
	- 
	\int_0^t 
	\nabla S_A(t-\tau) \cdot \big( v(\tau)\nabla K * v(\tau) \big) 
	\dta,
	\label{eq;GlobMild}
}
where the semigroup $\lbrace S_A(t) \rbrace_{t\ge 0}$ has been studied in Section \ref{sec:OpPr}.

\subsection{Local-in-time solutions}
We begin by a result on local-in-time solutions.

\begin{proposition}\label{prop;local-exist}
	For each $p$ satisfying condition \eqref{eqn:warP}
	and every $v_{0} \in L^p(\R^n)$,
	there exists $T>0$ and a unique local-in-time mild solution to problem \eqref{eq;linearized-DD} in 
	$C\big( [0,T); L^p(\R^n) \big)$.
\end{proposition}

For the proof of this proposition, 
one should follow the reasoning in the proof of Theorem~\ref{thm;main}. 
In particular, a solution is obtained via Proposition \ref{prop;Banach-fixed-pt} where the required estimate of the bilinear form 
\eq{
	\widetilde Q[v, w](t)
	=
	-\int_{0}^{t} \nabla S_A(t-s)\cdot \big(v(s)\nabla K * w(s)\big)(s) \ds
	\label{eq;BilQ}
}
is a direct consequence of the following lemma.

\begin{lemma}\label{lem:GradSemiEstConv}
	For each $p$ satisfying condition \eqref{eqn:warP}
	there exist $k\in (n, \infty]$ and positive numbers $C_1 = C_1(p,k,n,\nabla K) $ and $C_2 = C_2(p,k,n,\nabla K)$ such that
	\eq{
		\| 
		\nabla S_A(\tau) \cdot (v \nabla K *w) 
		\|_p 
		\le 
		C_2
		\tau^{-\frac{n}{2}{\frac{1}{k}-\frac{1}{2} }} 
		e^{C_1\tau} \|v\|_p \|w\|_p 
		\quad 
		\text{\rm for all} 
		\quad 
		v,w \in L^p(\R^n).
	}
	If $A<1$ then $C_1=0$ and if $A\ge 1$ then $C_1$ can be an arbitrary constant satisfying $C_1 > a = (\sqrt{A} - 1)^2$.
\end{lemma}

We skip the proof of this lemma, because it is the same as the proof of Lemma~\ref{lem;Duhamel-term-est}. 
In particular, 
it is based on the semigroup estimates from Lemma~\ref{lem;upperA}.

\subsection{Global-in-time solutions for $A \in [0,1)$}
The proof of Theorem \ref{thm;global-exist} requires 
the following extension of Lemma \ref{lem:GradSemiEstConv}.

\begin{lemma}\label{lem:GradConEst}
	Assume that $A \in [0,1)$. For
	\begin{itemize}
		\item
		each  $p$ satisfying  condition \eqref{eqn:warP},
		\item
		each $q>n$ if $n\ge 2$ and $q\ge 1$ if $n=1$ satisfying $q\in [p,\, 2p]$,
		\item
		each $r\ge 1$ such that $r \in \left[ \frac{q}{2}, \,p\right]$, 
	\end{itemize}
	there exists a constant $C > 0$ such that
	\begin{equation}\label{eq:GradConEst}
		\| 
		\nabla S_A(\tau) \cdot ( v \nabla K *w ) 
		\|_p 
		\le 
		C
		\tau^{-\frac{n}{2}( \frac{1}{r} - \frac{1}{p} ) - \frac{1}{2}} 
		\|v\|_{q} 
		\|w\|_{q},
	\end{equation}
	for all $v, w \in L^q(\R^n)$ and $\tau>0$.	
	
\end{lemma}
\begin{proof}
	First, 
	we use Theorem \ref{lem;Lp-Lq-estimate-S} and the H\"{o}lder inequality to estimate
	\eqsp{  
		\| 
		\nabla S_A(\tau) \cdot (v\nabla K *w ) 
		\|_p 
		\le 
		C
		\tau^{-\frac{n}{2}( \frac{1}{r} - \frac{1}{p} ) - \frac{1}{2}} 
		\| v \nabla K * w  \|_{r} 
		 \le 
		C 
		\tau^{-\frac{n}{2}( \frac{1}{r} - \frac{1}{p} ) - \frac{1}{2}} 
		\| v \|_{q} \| \nabla K * w \|_{k}
	}
	with $ 1 \le r \le p$ satisfying $\frac{1}{r} = \frac{1}{q} + \frac{1}{k}$.
	Next, we apply the Young inequality 
	\eqn{ 
		\| \nabla K *  w \|_{k} 
		\le 
		C\|\nabla K \|_{q_1} 
		\| w\|_{q} 
	}
	with 
	$
	\frac{1}{q_1} + \frac{1}{q} = 1 + \frac{1}{k}
	$.
	Let us show that $q_1 \in \left[1, \frac{n}{n-1} \right)$ in order to have $\nabla K \in L^{q_1}(\R^n)$. Indeed, by the assumption on $p,q,r$ we have
	\eqn{ 
		1 - \frac{1}{q} \le\frac{1}{q_1} 
		= 
		1 + \frac{1}{r} - \frac{2}{q} 
		\le 1,
	}
	where, for $n\ge 2$, we use also the inequality
	$1-\frac{1}{q} > 1-\frac{1}{n}$.
\end{proof}

\begin{proof}[Proof of Theorem \ref{thm;global-exist}]
	It is sufficient to construct a \git solution to problem~\eqref{eq;linearized-DD} formulated in the mild form \eqref{eq;GlobMild}, because $u = A + v$ by the uniqueness of solutions from Proposition \ref{prop;local-exist}. The solution is obtained via Proposition \ref{prop;Banach-fixed-pt} applied to equation~\eqref{eq;GlobMild} in the Banach space 
	\eq{
		\spl{
			\mathcal{X} =  C\big([0,\infty); L^p(\R^n)\big) 
			\cap 
			\bigg\{ 
			v\in C\big((0,\infty); L^q(\R^n)\big) 
			: 
			\sup_{t>0} 
			t^{\frac{n}{2}(\frac{1}{p} - \frac{1}{q})} 
			\| v(t)\|_{q} 
			< 
			\infty  
			\bigg\}
		}
	}
	with the norm 
	$
	\|v\|_\mathcal{X} 
	\equiv
	\sup_{t>0} \|v(t)\|_{p} 
	+ 
	\sup_{t>0} t^{\frac{n}{2}(\frac{1}{p} - \frac{1}{q})} 
	\| v(t)\|_{q}
	$. 
	
	First, 
	for every $v_0 \in L^p(\R^n)$, 
	it follows from Theorem~\ref{lem;Lp-Lq-estimate-S} that 
	\eq{\label{eq:Ku0Est}
		\| 
		S_A(t) v_0
		\|_\mathcal{X} 
		\le 
		C 
		\|v_0 \|_{p} 
	}
	for all $t>0$ and for some constant $C = C(p,q,r,n)>0$.
	
	Next, 
	we estimate the bilinear form $\widetilde Q[u,v]$ given by formula \eqref{eq;BilQ}
	for all $v,w \in \mathcal{X}$. Here, for $p$ satisfying condition \eqref{eqn:warP} with $p \le n$ and for $q \in \left(n, 2p\right]$, we choose one more exponent $r \in \left[ \frac{q}{2} , p \right]$ and $r \ge 1$ such that   
	\eqn{  
		\frac{2}{q} + \frac{1}{n} - \frac{1}{p} 
		\le 
		\frac{1}{r} \le \frac{2}{q}.
	} 
By Lemma \ref{lem:GradSemiEstConv}, 
	there exists $k > n$ such that
	\eqsp{\label{eq:QNormLocal}
		\| 
		\widetilde Q[v,w](t)
		\|_{p} 
		\le 
		C 
		\int_0^t 
		(t-s)^{-\frac{n}{2}{\frac{1}{k}-\frac{1}{2} }} 
		\|v(s)\|_p \|w(s)\|_p 
		\ds 
		\le 
		C
		t^{-\frac{n}{2 k} + \frac{1}{2}}
		\|
		v
		\|_\mathcal{X} 
		\|
		w
		\|_\mathcal{X} 
	}
	with $-\frac{n}{2 k} + \frac{1}{2} > 0$. 
	Next, applying Lemma~\ref{lem:GradConEst}, 
	we obtain
	\eq{
		\spl{
			\| 
		\widetilde	Q[v,w](t)
			\|_{p} 
			&\le 
			C 
			\int_0^t 
			(t-s)^{-\frac{n}{2}( \frac{1}{r} - \frac{1}{p} ) - \frac{1}{2}}
			\|v(s)\|_{q} 
			\|w(s)\|_{q} 
			\ds \\ 
			&\le 
			C 
			\|
			v
			\|_\mathcal{X} 
			\|
			w
			\|_\mathcal{X} 
			\int_0^t 
			(t-s)^{-\frac{n}{2}( \frac{1}{r} - \frac{1}{p} ) - \frac{1}{2}} 
			s^{-n( \frac{1}{p} - \frac{1}{q} )} 
			\ds.
			\label{eq:BformEstimate}
		}
	}
	Notice that the integral on the right hand side is convergent because the exponents satisfy the inequalities 
	\eqsp{
		\quad
		-\frac{n}{2}
		\Big(
		\frac{1}{r}-\frac{1}{p}
		\Big)
		-
		\frac{1}{2} 
		=&
		-1 
		+
		\frac{n}{2}
		\Big(
		\frac{1}{p} - \frac{1}{q}
		\Big) 
		>  -1, 
		\\
		-n
		\Big(
		\frac{1}{p}-\frac{1}{q}
		\Big)
		>
		&-n 
		\Big(
		\frac{2}{q} - \frac{1}{q}
		\Big)
		=-\frac{n}{q} > -1
	}
	resulting immediately from the properties of the parameters $p,q,r$.
	Therefore,
	\eq{\label{eq:QNormGlobal}
		\spl{
			\|
		\widetilde	Q[v,w](t)
			\|_{p} 
			\le &
			C 
			B
			\Big( 
			1 
			- 
			n 
			\Big( 
			\frac{1}{p} - \frac{1}{q} 
			\Big), 
			\frac{1}{2} 
			- 
			\frac{n}2 
			\Big( 
			\frac{1}{r} 
			- 
			\frac{1}{p} 
			\Big) 
			\Big)
			\\ 
			& \times 
			t^{-\frac{n}{2}( \frac{1}{r} - \frac{1}{p} ) - \frac{1}{2} - n ( \frac{1}{p} - \frac{1}{q} ) + 1} 
			\|
			v
			\|_\mathcal{X} 
			\|
			w
			\|_\mathcal{X} 
			, 
		}
	} 
	where $B = B(x,y)$ denotes the beta function. 
	Moreover, 
	by the assumption on $r$, 
	we have 
	\eqspn{
		-\frac{n}{2}
		\Big( 
		\frac{1}{r} - \frac{1}{p} 
		\Big) 
		- 
		\frac{1}{2} 
		- 
		n 
		\Big( 
		\frac{1}{p} - \frac{1}{q} 
		\Big) 
		+ 
		1 
		&= 
		-
		\frac{n}{2r} 
		-
		\frac{n}{2p} 
		+
		\frac{n}{q} 
		+ 
		\frac{1}{2} \\ 
		&\le 
		-
		\frac{n}{2}
		\Big( 
		\frac{2}{q} 
		+ 
		\frac{1}{n} 
		- 
		\frac{1}{p} 
		\Big) 
		-
		\frac{n}{2p} 
		+
		\frac{n}{q} 
		+
		\frac{1}{2} 
		=
		0.
	}
	Consequently, 
	inequalities \eqref{eq:QNormGlobal} and \eqref{eq:QNormLocal} provide the estimate
	\eqsp{\label{eq:BFtEst}
		\sup_{t>0} 
		\| 
		\widetilde Q[v,w] (t) 
		\|_{p} 
		\le &
		C 
		\sup_{t>0} 
		\big(
		\min
		\big\{
		t^{-\frac{n}{2 k} + \frac{1}{2}}, 
		\, 
		t^{-\frac{n}{2}( \frac{1}{r} - \frac{1}{p}) -\frac12 - n ( \frac{1}{p} - \frac{1}{q} ) + 1}
		\big\} 
		\big)
		\|
		v
		\|_\mathcal{X} \|w\|_\mathcal{X}\\
		\le &
		C 
		\|
		v
		\|_\mathcal{X}
		\|
		w
		\|_\mathcal{X}
	}
	with a positive constant $C>0$.
	
	We proceed with the $L^q$-component of the norm in $\mathcal{X}$ analogously. By Lemma~\ref{lem:GradSemiEstConv},
	\eqsp{\label{eq:QNormLocalQ} 
		t^{\frac{n}{2} ( \frac{1}{p} - \frac{1}{q} ) } 
		\|
		\widetilde Q[v,w](t)
		\|_{q} 
		\le 
		C
		t^{\frac{n}{2p}-\frac{n}{2}(\frac{1}{q} + \frac{1}{k}) + \frac{1}{2}}
		\|
		u
		\|_\mathcal{X}
		\|
		v
		\|_\mathcal{X} 
	}
	with 
	$
	\frac{n}{2p}-\frac{n}{2}\big(\frac{1}{q} + \frac{1}{k}\big) + \frac{1}{2} 
	\ge 
	0
	$ 
	and, by Lemma \ref{lem:GradConEst},   
	\eq{\label{eq:BFtPowEst}
		\spl{
			t^{\frac{n}{2} ( \frac{1}{p} - \frac{1}{q}) } 
			\|
			\widetilde Q[v,w](t)
			\|_{q} 
			&\le 
			C 
			t^{\frac{n}{2}( \frac{1}{p} - \frac{1}{q})} 
			\|
			v
			\|_\mathcal{X} 
			\|
			w
			\|_\mathcal{X} 
			\int_0^t 
			(t-s)^{-\frac{n}{2} ( \frac{1}{r} - \frac{1}{q}) - \frac{1}{2}} 
			s^{-n( \frac{1}{p} -\frac{1}{q} )} 
			\ds \\ 
			&=
			C 
			t^{ -\frac{n}{2}( \frac{1}{r} - \frac{1}{q}) - \frac{1}{2} - \frac{n}{2} ( \frac{1}{p} - \frac{1}{q} ) + 1}
			\| 
			v
			\|_\mathcal{X} 
			\|
			w
			\|_\mathcal{X}
		}
	}
	with 
	$$
	-\frac{n}{2}
	\Big( 
	\frac{1}{r} - \frac{1}{q} 
	\Big) 
	- 
	\frac{1}{2} 
	- 
	\frac{n}{2} 
	\Big( 
	\frac{1}{p} - \frac{1}{q} 
	\Big) 
	+ 1
	\le 
	0.
	$$ 
	Therefore, 
	there exists a constant $C>0$ independent of $t>0$ such that 
	\eqsp{\label{eq:LqMinEst}
		&
		\sup_{t>0}
		\, t^{\frac{n}{2}(\frac{1}{p} - \frac{1}{q} )} 
		\| 
		\widetilde Q[v,w](t) 
		\|_{q} 
		\\ 
		\le &
		C 
		\sup_{t>0}  
		\big(
		\min
		\big\{ 
		t^{\frac{n}{2p}-\frac{n}{2}(\frac{1}{q} + \frac{1}{k}) + \frac{1}{2}}
		, 
		t^{-\frac{n}{2}( \frac{1}{r} - \frac{1}{q} ) - \frac{1}{2} - \frac{n}{2} ( \frac{1}{p} - \frac{1}{q} ) + 1}
		\big\}
		\big)
		\|
		v
		\|_\mathcal{X} 
		\|
		w
		\|_\mathcal{X}\\
		\le &
		C 
		\|
		v
		\|_\mathcal{X} 
		\|
		w
		\|_\mathcal{X}.
	}
	Finally, 
	it follows from inequalities \eqref{eq:BFtEst} and \eqref{eq:LqMinEst} that
	\eq{
		\| 
	\widetilde	Q[v,w]
		\|_\mathcal{X} 
		\le 
		\eta 
		\|
		v
		\|_\mathcal{X} 
		\|
		w
		\|_\mathcal{X}
	} 
	for a positive number $\eta$ independent of $t, \, v$ and $w$. 
	Hence,  
	if $\| v_0\|_{p}$ is sufficiently small, 
	by inequality \eqref{eq:Ku0Est} and Proposition \ref{prop;Banach-fixed-pt}, there exists a mild solution of problem~\eqref{eq;linearized-DD} in the space $\mathcal{X}$. 
	This solution is unique by Proposition \ref{prop;local-exist}.
\end{proof}

\subsection{Instability for $A>1$}
In this section, we prove that the constant solution $u = A$ of problem \eqref{eq;DD}  is unstable in $L^p(\R^n)$ if $A>1$. Here we apply the classical  so-called linearization principle which was used \textit{e.g.} in 
\cite{Friedlander,ShSt}.

\begin{proof}[Proof of Theorem \ref{thm;instability}]
	We begin with arbitrary $\delta \in (0,1)$ and arbitrary $v_0 \in L^p(\R^n)$ with $\|v_0\|_p = 1$ to be chosen later on. Under the assumpion on $p$,
	by Corollary \ref{lem;LitSol}, 
	there exists a unique local-in-time mild solution $v \in C\big([0,T_{\max}); L^p(\R^n)\big) $ to problem~\eqref{eq;linearized-DD}, 
	with the initial datum $\delta v_0$. 
	Suppose that this solution is \git and for $a = ( \sqrt{A} - 1 )^2$ define two numbers
	\eqsp{
		T 
		= 
		\sup 
		\left\lbrace 
		t
		: 
		\, 
		\| 
		v(\tau) - S_A(\tau) \delta v_0 
		\|_p 
		\le 
		\frac{\delta}2 
		e^{a\tau} 
		\text{ for all } 
		\tau\in[0,t] 
		\right\rbrace 
		\label{eq:TDef}
	}
	and 
	$
	T' 
	= 
	\frac{1}{a} 
	\log 
	\left(
	\frac{2}{\delta}
	\right),
	$ 
	hence 
	$\delta e^{aT'} = 2$. 
	
	If either $T > T'$ or $T= \infty$, 
	then the zero solution is unstable.
	Indeed, 
	by Lemma \ref{lem:SMis}, for each $\gamma >0$ we may choose $v_0\in L^p(\R^n)$ 
	with $ \| v_0\|_p = 1$ 
	such that
	\eqsp{\label{eq:InRsT}
		\big\|
		S_A(T')\delta v_0 
		- 
		e^{aT'} \delta v_0 
		\big\|_p
		\le 
		\gamma 
		\delta 
		\|
		v_0
		\|_p
		=
		\gamma \delta
		.
	}
	By the definition of $T$ and by inequality \eqref{eq:InRsT}, we obtain 
	\eqsp{
		\|
		v(T')
		\|_p 
		\ge 
		\|
		S_A(T')
		\delta 
		 v_0
		\|_p 
		- 
		\frac{\delta}{2} 
		e^{aT'} 
		\ge 
		\| 
		e^{aT'} \delta v_{0} 
		\|_{p}
		-
		\gamma \delta 
		- 
		\frac{\delta}{2} 
		e^{aT'} 
		= 
		\frac\delta2 
		e^{aT'} 
		- 
		\gamma 
		\delta 
		\ge 
		1 
		- 
		\gamma.
	} 
	In particular, $	\|	v(T') 	 \|_p \ge\frac{1}{2}$ for $\gamma= \frac{1}{2}$.
	
	Next, suppose that $T \le T'$ and consider the mild representation of the solution of problem \eqref{eq;linearized-DD} with the initial condition $\delta v_0$
	\eqsp{
		v(t) 
		- 
		S_A(t)
		\delta 
		v_0 
		= 
		\int_0^t 
		\nabla 
		S_A(t-\tau) 
		\cdot 
		\big( 
		v(\tau)\nabla K * v(\tau) 
		\big) 
		\dta.
	}
	Lemma \ref{lem:GradSemiEstConv} with $C_1 = \frac{3}{2} a$, 
	estimates \eqref{eq:VEs} and definition of $T$ in \eqref{eq:TDef} lead to the inequality
	\eqsp{
		& \| 
		v(t)  -  S_A(t)\delta v_0 
		\|_p \\
		\le &
		C 
		\int _0^t 
		(t-\tau)^{-\frac{n}{2}{\frac{1}{q}-\frac{1}{2} }} 
		e^{\frac{3}{2}a(t-\tau)} 
		\|
		v(\tau)
		\|_p^2  
		\dta  \\
		\le &
		C 
		\int_0^t 
		(t-\tau)^{-\frac{n}{2}{\frac{1}{q}-\frac{1}{2} }} 
		e^{\frac{3}{2}a(t-\tau)} 
		\left(
		\| 
		S_A(\tau)
		\delta 
		v_0 
		- 
		v(\tau)  
		\|_p^2 
		+ 
		\| 
		S_A(\tau)
		\delta 
		v_0
		\|_p^2 
		\right) \, {\rm d}\tau\\
		\le &
		C 
		\int_0^t 
		(t-\tau)^{-\frac{n}{2}{\frac{1}{q}-\frac{1}{2} }} 
		e^{\frac{3}{2}a(t-\tau)} 
		\Big(
		\frac{\delta^2}{4}e^{2a\tau} 
		+ 
		4\delta^2e^{2a\tau} 
		\Big) 
		\dta\\
		\le &
		C 
		\delta^2 
		e^{2at}. 
		\label{eq:DuhDifEs}
	}
	for all $t\in[0, T]$, 
	where the last inequality is explained in Remark \ref{rem:IntTT}, below. 
	Thus, from the definition of the number $T$ and from inequality \eqref{eq:DuhDifEs} for $t = T$, 
	we have the relations
	\eqspn{
		\frac\delta2 
		e^{aT} 
		= 
		\| 
		v(T) 
		- 
		S_A(T) \delta v_0 
		\|_p 
		\le 
		C
		\delta^2 
		e^{2aT}		
	}
	which imply the inequality $\frac{1}{2C} 
	\le 
	\delta 
	e^{aT}$. 
	In particular, 
	the number $T^*$ defined by the equation  
	$
	\delta e^{aT^*} 
	= 
	\frac{1}{2C}
	$
	satisfies $T^* \leq T$. 
	Hence, by inequality \eqref{eq:DuhDifEs} with $t=T^*$ we have 
	\eqsp{
		\left\| 
		v(T^*) 
		\right\|_p
		\ge
		\|
		S_A(T^*)
		\delta  
		v_0
		\|_p 
		-
		\frac{1}{2}\delta e^{aT^{*}}
		= 
		\|
		S_A(T^*)
		\delta  
		v_0
		\|_p 
		-
		\frac{1}{4C} .
	}
	Finally, we apply Lemma \ref{lem:SMis} 
	with 
	$\gamma = \frac{1}{ 4k_{0}C} \le 1$ 
	for some fixed $k_{0} \gg 1$
	and $T'$ (recall that $T^* \leq T \leq T'$) in order to obtain $v_0 \in L^p(\R^n)$ with $\|v_0\|_p = 1$ such that 
	\eqsp{
		\left\| 
		v(T^*) 
		\right\|_p
		\ge 
		\|
		S_A(T^*)
		\delta  
		v_0
		\|_p 
		-
		\frac{1}{4C} 
		\ge 
		\delta 
		e^{aT^*} 
		- 
		\gamma 
		\delta
		- 
		\frac{1}{4C}
		\ge
		\frac{1}{4C}
		-\gamma 
		=
		\frac{1}{4C}
		\Big(
		 1-\frac{1}{k_{0}}
		\Big)
		.
	 }
	The proof of instability is completed because the right hand side is independent of 
	$\delta$.
\end{proof}

\begin{remark}\label{rem:IntTT}
	The last inequality in \eqref{eq:DuhDifEs} follows from a direct calculation which we present for the reader convenience. 
	For a fixed $\eta \in (0, T)$, 
	we obtain
	\eqspn{
		&\quad e^{\frac{3}{2}at}
		\int_0^t 
		(t-\tau)^{-\frac{n}{2}{\frac{1}{q}-\frac{1}{2} }} 
		e^{\frac{1}{2}a\tau} 
		\dta \\
		&= 
		e^{\frac{3}{2}at}
		\left(
		\int_0^{t-\eta} 
		(t-\tau)^{-\frac{n}{2}{\frac{1}{q}-\frac{1}{2} }} 
		e^{\frac{1}{2}a\tau}
		\dta 
		+ 
		\int_{t-\eta}^{t} 
		(t-\tau)^{-\frac{n}{2}{\frac{1}{q}-\frac{1}{2} }} 
		e^{\frac{1}{2}a\tau}  
		\dta
		\right) 
		\\ 
		&\le 
		e^{\frac{3}{2}at} 
		\left( 
		\eta^{-\frac{n}{2}\frac{1}{q}-\frac{1}{2}}
		e^{\frac{1}{2}a(t-1)} 
		+
		e^{\frac{1}{2}at}
		\int_{t-\eta}^{t} 
		(t-\tau)^{-\frac{n}{2}{\frac{1}{q}-\frac{1}{2} }} 
		\dta 
		\right) \\
		&\le 
		C
		e^{2at}.
	}
\end{remark}


\section*{Acknowledgments}
	S.\ Cygan and H.\ Wakui were supported by the Polish NCN grant \hbox{2016/23/B/ST1/00434}. H.\ Wakui was also supported by JSPS Grant-in-Aid for JSPS Fellows Grant number JP20J00940.


\end{document}